\documentclass[10pt,a4paper]{article}
\usepackage[utf8]{inputenc}
\usepackage[english]{babel}
\usepackage{geometry}
\usepackage{amsmath}
\usepackage{amsfonts}
\usepackage{amssymb}
\usepackage{amsthm}
\usepackage{bm}
\usepackage{hyperref}
\usepackage{kpfonts}
\usepackage{enumerate}
\usepackage{esint}
\usepackage[bb=dsserif]{mathalpha}
\usepackage{todonotes}
\usepackage{xcolor}
\usepackage{csquotes}
\usepackage[backend=biber, giveninits=true, maxbibnames=7]{biblatex}

\newcommand{\gmax}{g_{\mathrm{max}}}
\newcommand{\loc}{\mathrm{loc}}
\newcommand{\NN}{\mathbb{N}}
\newcommand{\RR}{\mathbb{R}}

\newcommand{\divg}{\mathrm{div}}
\newcommand{\supp}{\mathrm{supp}}
\newcommand{\faskip}{\hspace{8pt}}

\newtheorem{thm}{Theorem}[section]
\newtheorem{lem}[thm]{Lemma}
\newtheorem{prop}[thm]{Proposition}
\newtheorem{cor}[thm]{Corollary}

\newtheorem{mainthm}{Theorem}

\numberwithin{equation}{section}

\title{On the shape of the positivity region for a free boundary problem describing cell polarization}
\author{Sebastián Flores Sepúlveda\thanks{Universität Bonn, Germany, \texttt{sfloress@uni-bonn.de}}, Barbara Niethammer\thanks{Universität Bonn, Germany, \texttt{niethammer@iam.uni-bonn.de}}, Juan J. L. Velázquez\thanks{Universität Bonn, Germany, \texttt{velazquez@iam.uni-bonn.de}}}

\begin{document}

\maketitle

{\small \textbf{Abstract}
	In this paper we study a mass-constrained free boundary problem modeling cell polarization, in the regime where the mass is small. In the generic case of a signal with nondegenerate maxima, we prove that the solution converges locally to a global, integrable solution to an obstacle problem in the plane. We further show that the interface of the solution to the limit problem is an ellipse, the equation of which is explicit. We also study some cases where the signal has degenerate maxima, highlighting a variety of possible behaviors.
}

\section{Introduction}

Cell polarization is a crucial step in a number of biological processes, including cell division, differentiation and motility. There exists a vast biological literature studying this phenomenon, and a number of mathematical models have been proposed for its description. 

In the model proposed in \cite{NRV20}, cell polarization is described by a bulk-surface reaction-diffusion system. In short, inactive signaling proteins diffuse in the cytosol, represented as a bounded domain $ \Omega $ in three-dimensional space, and can attach to the membrane, i.e. the boundary of $ \Omega $, where they can become active by reacting with molecules in the exterior of the cell, represented by a ``forcing'' term in the reaction-diffusion system. 

In a suitable scaling limit, this system may be reduced to an obstacle-type problem for the density of active signaling proteins on the membrane, with a mass constraint representing the total number of signaling proteins. The rigorous justification of this reduction as well as existence and uniqueness of steady states for the obstacle-type problem may be found in \cite{NRV20} and further analysis of the time-dependent problem may be found in \cite{LNRV21, LNRV23, LNRV25}.

In this work, we study the model from \cite{NRV20} in the regime where the number of signaling proteins is small. We look for a detailed description of the support and the free boundary of the solutions to the obstacle-type problem posed by the model.  

The problem is as follows. Let $\Omega \subset \RR^3$ be a smooth bounded domain, $\Gamma:= \partial \Omega$, and $ g:\Gamma \rightarrow (0,1) $ be a continuous function. Then the density of active signaling proteins $ u $ solves 
\begin{equation}
\label{eq:obstacle_gamma}
\begin{cases}
	-\Delta_{\Gamma} u = -(1-g)\xi + \alpha g & \mathrm{on} ~ \Gamma, \\ 
	\xi \in [0,1],~~u \xi  = 1 & \mathrm{on}~ \{u>0\},\\
	u\geq 0,~ \displaystyle\int_{\Gamma} u \mathrm{d}S = M,
\end{cases}
\end{equation}
where $\Delta_{\Gamma}$ is the Laplace-Beltrami operator, and $ g $ represents a normalized chemical signal. The existence of a unique solution to \eqref{eq:obstacle_gamma}, given $M$, is given by Theorem 3.7 in \cite{NRV20}. We will denote by $(u_M, \xi_M, \alpha_M)$ the triplet which satisfies \eqref{eq:obstacle_gamma}. It is possible to think of the time-dependent version of problem \eqref{eq:obstacle_gamma} as a mass conserving parabolic obstacle problem, see \cite{LNRV23}.

Furthermore, \eqref{eq:obstacle_gamma} may be written in a more compact form by using the representation formula for $ \xi_M $:
\[
	\xi_M = \begin{cases}
		1			& \text{in }\{ u_M > 0 \},\\
		\dfrac{\alpha_M g}{1-g} & \text{in }\{ u_M = 0 \}.
	\end{cases}
\]
therefore, $ u_M $ satisfies
\[
\begin{cases}
	-\Delta_{\Gamma} u_M \geq ((1+\alpha_M)g - 1) & \mathrm{on} ~ \Gamma, \\ 
	u_M \geq 0 & \text{on }\Gamma, \\
	-\Delta_{\Gamma} u_M = ((1+\alpha_M)g - 1)& \mathrm{on} ~ \{u_M > 0\},\\
	\displaystyle\int_{\Gamma} u_M \mathrm{d}S = M.
\end{cases}
\]

There is a further reformulation of the problem which highlights its nonlocality. Exploiting an explicit formula for $\alpha_M$ (see \cite{LNRV21}), the problem may be rewritten as
\[
\begin{cases}
	-\Delta_{\Gamma} u_M = \left(1 - \dfrac{g}{\fint_{\{u_M > 0\}} g \mathrm{d}S}\right) \mathbb{1}_{\{u_M > 0 \}}& \mathrm{on} ~ \Gamma, \\ 
	u_M \geq 0 & \text{on }\Gamma, \\
	\displaystyle\int_{\Gamma} u_M \mathrm{d}S = M,
\end{cases}
\]
where $\fint_A g \mathrm{d}S$ denotes the average over the set $A\subset \Gamma$.

We shall see that the behavior of $ (u_M, \alpha_M) $ for $ M $ small depends on the structure of $ g $ near its maximum points. In the generic case where $ g $ is a Morse function, we may informally state our main result as follows. As in \cite{NRV20}, we denote $ \gmax := \max_{\Gamma} g $ and
\begin{equation}
\label{eq:alphazero}
	\alpha_0 := \frac{1 - \gmax}{\gmax}.
\end{equation}
\begin{mainthm}
	\label{mainthm:morse}
	Let $ g:\Gamma \rightarrow \RR $ be of class $ C^2 $ with a finite set of nondegenerate maximum points $ \{p_i\}_{i=1}^N $, i.e. $ g(p_i) = \gmax $ and $ D^2g(p_i) $ is a negative definite bilinear form for every $ i=1,\hdots,N $. Then, there exist positive numbers $ \{\lambda_i\}_{i=1}^N $ with $ \sum_{i=1}^{N} \lambda_i = 1 $ such that
	\begin{align*}
	\frac{1}{M}u_M \rightarrow \sum_{i=1}^{N} \lambda_i \delta_{p_i} \faskip \text{ as } M\rightarrow 0,
	\end{align*}
	where $ \delta_{p_i} $ is a Dirac delta centered at $ p_i $ and each $ \lambda_i $ depends only on $ D^2 g(p_i) $. Moreover, under a suitable scaling and near a maximum point $ p $, the free boundary $ \partial \{ u_M > 0\} $ converges to an ellipse, and is a smooth curve locally near $p$, for $ M $ sufficiently small.
\end{mainthm}

The formula for $ \lambda_i $ is given in Theorem \ref{thm:lambda_i}. 

The scaling mentioned in Theorem \ref{mainthm:morse} is the consequence of a local analysis around a maximum point. In fact, we perform this analysis under the slightly more general case when $\gmax - g$ is a positive function, homogeneous of degree $\gamma > 0$. We obtain the following theorem.

\begin{mainthm}
	\label{mainthm:isolated_hom}
	Let $ p \in \Gamma $ be a maximum of $ g $. Assume that there exist a neighborhood $\Theta$ of $ p $ and functions $ f,r: \Theta \rightarrow \RR $ such that the following holds in any coordinate system in $ \Theta $ mapping $ p $ to 0:
	\begin{enumerate}
		\item $ f $ is of class $ C^{1} $, $f(0) = 0$, $ f(x) > 0 $ for any $ x\neq 0 $ and $ f $ is positively $ \gamma $-homogeneous for some $ \gamma > 0 $, i.e. $ f(\lambda x) = \lambda^{\gamma} f(x) $ for any $ \lambda \geq 0 $,
		\item $ \frac{|r(x)|}{|x|^{\gamma}} \rightarrow 0$ as $ x\rightarrow 0 $, where $ \gamma $ is the homogeneity degree of $ f $,
		\item for any $ x \in \Theta $,
		\begin{equation}
		\label{eq:hyp_hom}
		g(x) = \gmax - f(x) + r(x).
		\end{equation}
	\end{enumerate}
	Then, defining $ m := \displaystyle\int_{\Theta} u_M \mathrm{d}S $ and 
	\[
		U_m(y) = m^{-(\gamma + 2)/(\gamma + 4)} u_M(m^{1/(\gamma + 4)}y) \faskip \forall y \in m^{-1/(\gamma + 4)} \Theta,
	\]
	it holds that  
	\begin{align*}
		\frac{\alpha_{M} - \alpha_0}{m^{\gamma/(\gamma + 4)}} &\rightarrow \overline{\alpha}, \\
		U_{m} &\rightarrow \overline{U} \text{ in }C^{1,\delta}(\RR^{2}) \faskip \forall \delta \in (0,1), \\ 
		\partial \{U_m > 0\} &\rightarrow \partial \{\overline{U} > 0\} \faskip \text{in the Hausdorff distance.}
	\end{align*}
	Moreover, $ \overline{U} $  is the unique solution to 
	\begin{equation}
		\label{eq:global}
		\begin{cases}
			-\Delta \overline{U} = \left( \gmax\overline{\alpha} - \dfrac{1}{\gmax}f(y) \right) \mathbb{1}_{\{\overline{U} > 0\}} & \text{ in }\RR^{2},\\
			\overline{U} \geq 0 \faskip \text{ in } \RR^{2}, \faskip \displaystyle\int_{\RR^{2}} \overline{U} \mathrm{d}x = 1, 
		\end{cases}
	\end{equation}
	and has compact support.
\end{mainthm}

Theorem \ref{mainthm:isolated_hom} implies that the mass of $ U_M $ concentrates around the points of maximum homogeneity degree $ \gamma $. Indeed, let $ p_1 $, $ p_2 $ be two maximum points with homogeneity degree $ \gamma_i $, $ i=1,2 $, with $ \gamma_2 > \gamma_1 $. Then, we have
\[
	\frac{\alpha_M - \alpha_0}{m_i^{\gamma_i / (\gamma_i + 4)}} \rightarrow \overline{\alpha}_i \faskip \text{for }i=1,2.
\]
This implies that 
\[\frac{m_1^{\gamma_1/(\gamma_1 + 4)}}{m_2^{\gamma_2/(\gamma_2 + 4)}} \rightarrow \frac{\overline{\alpha}_2}{\overline{\alpha}_1},\]
in particular, there is some $ C > 0 $ such that
\[
	C^{-1} m_2^{\frac{\gamma_2(\gamma_1 + 4)}{\gamma_1(\gamma_2 + 4)}} \leq m_1 \leq C m_2^{\frac{\gamma_2(\gamma_1 + 4)}{\gamma_1(\gamma_2 + 4)}}.
\]
Thus, since $ \gamma_2 > \gamma_1 $, $ m_1 = o(m_2) $ as $ M \rightarrow 0$.

In the generic case of a Morse function, in the language of Theorem \ref{mainthm:isolated_hom}, we have $ f(y) = y^{\top}Ay $ which is homogeneous of degree 2, hence 
\[u_M(x) \sim m^{2/3} U_m(m^{-1/6}x)\]
near maximum points. Theorem \ref{mainthm:isolated_hom} also implies that the masses near maximum points are all comparable, which in turn yields the existence of the $ \lambda_i $ in Theorem \ref{mainthm:morse}. 

Solving problem \eqref{eq:global} explicitly for quadratic $f$, we get the limit ellipses mentioned in Theorem \ref{mainthm:morse}. Ellipsoidal and paraboloidal shapes are the only possible ones in several free boundary problems related to the obstacle and Stefan problem, as was first noticed in the physical literature \cite{Iva47,HC61}. Classification of all global solutions in the two-dimensional obstacle problem was obtained in \cite{Sak81}, and the study of higher dimensional case, in the context of global solutions for the Stefan problem, was initiated in \cite{DF86}. In both of these works, ellipsoidal shapes play a key role in the analysis. In the recent works \cite{ESW23, EFW25, SW25}, a complete classification of global solutions to the elliptic obstacle problem in all dimensions is achieved. We may read our results as a classification result for bounded solutions to
\[-\Delta u = \left( 1 - x^{\top}A x \right)\mathbb{1}_{\{u>0\}}. \]
Using the explicit formulas for the solution to this problem we also prove that the right-hand side does not vanish near the free boundary, allowing us to apply the regularity results for the free boundary of Caffarelli \cite{Caf77} for small but positive mass $M$. 

As stated, Theorems \ref{mainthm:morse} and \ref{mainthm:isolated_hom} completely describe the generic case of Morse functions. Describing the behavior of $ u_M $ near degenerate maxima, i.e. maxima where $D^2 g(p)$ is a singular bilinear form, is a delicate issue. We pick some examples in order to illustrate the possible situations. In each case, the structure of $ g $ is given by a higher order Taylor expansion of $ g $ which, since we assume the maxima are isolated, must eventually contain a polynomial which is positive outside $ 0 $. 

First, we study an example where $ f $ is such that \eqref{eq:hyp_hom} holds, but the homogeneity of $ f $ in each variable is different. This yields a degenerate elliptic operator in the limit, so that regularity estimates cease to imply compactness of the blow-up sequences. Nevertheless, we prove the following.
\begin{mainthm}
	\label{mainthm:inhom}
	Suppose $ p \in \Gamma $ is an isolated maximum of $ g $ and there exists a neighborhood $ \Theta $ of $ p $ such that, in local coordinates,
	\begin{equation}
		\label{eq:inhom_g}
		g(x) = \gmax - ax_1^{4} - bx_2^{2} + r(x),
	\end{equation}
with $ a $, $ b >0 $ and $ r(x) = o(x_1^{2} + x_2^{4}) $. Then, there exists a coordinate chart $ (x,\Theta) $ around $ p $ such that $x(p) = 0$, a number $\overline{\alpha} > 0$ and a function $ \overline{v} \in C_c^{1,1}(\RR^{2}) $ such that the following holds. Define 
	\[
		U_{m}(y) = m^{-8/11} u_M(x^{-1}(m^{1/11} y_1, m^{2/11} y_2)), \faskip \forall y\in m^{-8/11} x(\Theta),
	\]
	then
	\begin{align*}
		\frac{\alpha_M - \alpha_0}{m^{4/11}} &\rightarrow \overline{\alpha}, \\
		U_{m} &\rightarrow \overline{U} \text{ strongly in } L^{p}(\RR^{2}) \faskip \forall p\in [1,\infty], \\
		\partial_2 U_{m} &\rightharpoonup \partial_2 \overline{U} \text{ weakly in } L^{2}(\RR^{2}), \\
		\partial \{U_m > 0\} &\rightarrow \partial \{\overline{U} > 0\} \faskip \text{in the Hausdorff distance.}
	\end{align*}
	Moreover, $ \overline{U} $ is the unique solution to 
	\[
	\begin{cases}
		-\partial_{22} \overline{U} = \left( \gmax \overline{\alpha} -\dfrac{1}{\gmax}( a x_1^{4} + b x_2^{2}) \right)\mathbb{1}_{\{ \overline{U} > 0 \}} & \text{ in }\RR^{2}, \\ 
		\overline{U} \geq 0 \faskip \text{in }\RR^{2}, \faskip \displaystyle \int_{\RR^{2}} \overline{U} \mathrm{d}x  = 1, 	& 
	\end{cases}	
	\]
	and has compact support.
\end{mainthm}

A second example in which the hypotheses of Theorem \ref{mainthm:isolated_hom} fail is given by
\begin{equation}
	\label{eq:g_noncoercive}
	g(x) = \gmax - a x_1^{4} - b x_1^{2}x_2^{2} - c x_2^{6} + r(x)
\end{equation}
around a maximum, with $ a $, $ b $, $ c > 0 $. In this particular case, $ f $ may be chosen as positively 4-homogeneous, but its zero set is nontrivial. Furthermore, one could consider two different scales, $ (x_1, x_2) \sim (\ell y_1, \ell y_2) $ and $ (x_1, x_2) \sim (\ell^2 y_1, \ell y_2) $ for some characteristic length $ \ell $, which would lead, a priori, to two different possible blow-ups. We will prove that this is not the case, indeed, the following holds.

\begin{mainthm}
	\label{mainthm:noncoercive}
	Assume $ p\in \Gamma $ is a maximum of $ g $ such that \eqref{eq:g_noncoercive} holds in a neighborhood of $ p $. Define the rescaling
\[
	u_M(x) = m^{3/4} U_m(m^{-1/8}x),
\]
Then, there exist $ \overline{\alpha} > 0 $ and a nontrivial function $ \overline{U} \in W^{2,p}_{\loc}(\RR^{2}) $ for all $ p \in [1, \infty) $ such that 
	\begin{align*}
		\frac{\alpha_{M} - \alpha_0}{m^{1/2}} &\rightarrow \overline{\alpha}, \\
		U_{m} &\rightarrow \overline{U} \text{ in }C^{1,\delta}(\RR^{2}) \faskip \forall \delta \in (0,1),
	\end{align*}
\[
	\begin{cases}
		-\Delta \overline{U} = \left( \gmax\overline{\alpha} - \dfrac{1}{\gmax}(a x_1^{4} + bx_1^{2}x_2^{2}) \right) \mathbb{1}_{\{\overline{U} > 0 \}} & \text{ in }\RR^{2},\\
		\overline{U} \geq 0 \faskip \text{in }\RR^{2}, \faskip \displaystyle \int_{\RR^{2}} \overline{U} \mathrm{d}x  = 1. 	& 
	\end{cases}
\]
\end{mainthm}

These two examples do not exhaust all possible degeneracies, even if we assume all global maxima of $ g $ to be isolated.

The structure of the paper is as follows. In Section \ref{sect:preliminaries} we collect some results which will be necessary in our analysis. In Section \ref{sect:hom} we prove Theorem \ref{mainthm:isolated_hom} by performing a local analysis around maxima described by a positive, homogeneous function. In Section \ref{ssect:quadratic} we prove Theorem \ref{mainthm:morse} by studying the special case of nondegenerate maxima, and further describe the case where all maxima of $g$ are nondegenerate, obtaining explicit formulas for the limit profile. In Sections \ref{sect:inhom} and \ref{sect:noncoercive}, we prove Theorems \ref{mainthm:inhom} and \ref{mainthm:noncoercive} respectively. In Section \ref{sect:misc} we discuss, at a formal level, some tools for the analysis of the behavior of the interfaces in other cases.

\section{Preliminaries}
\label{sect:preliminaries}

In what follows, $\Omega \subset \RR^{3}$ is a bounded domain, whose boundary $\Gamma$ is $C^5$ regular. Recall that the Laplace Beltrami operator may be written in local coordinates as
\[-\Delta_{\Gamma}u = -\divg(H(x) \nabla u(x)),\]
where $H(x)$ is a symmetric matrix whose coefficients are those of the metric of $\Gamma$. By the regularity of $\Gamma$, the components of $H$ are $C^4(\Gamma)$ functions and, by choosing a suitable coordinate chart (for example, the Riemannian normal coordinates, see \cite[Section 5.5]{Pet16}), we may assume
\begin{equation}
	\label{eq:metric}
	H(x) = I + O(|x|^{2})\faskip \text{as } |x| \rightarrow 0. 
\end{equation}

Let us recall equivalent formulations of the obstacle problem. All of these formulations are classical, see \cite{FRRO22}. 
\begin{prop}
	Let $ \Omega $ be a smooth bounded domain in $ \RR^{n} $, $n\geq 2$ and $ f \in C(\overline{\Omega}) $. Then, for $ u \in H_0^{1} (\Omega)$ nonnegative, the following are equivalent,
	\begin{enumerate}
		\item $ u $ satisfies
			\[\int_{\Omega} \nabla u \cdot \nabla(v - u) \mathrm{d}x \geq \int_{\Omega} f(v-u) \mathrm{d}x \faskip \forall v \in H_0^{1}(\Omega), ~ v\geq 0.  \]  
		\item $ u $ minimizes 
		\[J(v) = \int_{\Omega} \frac{|\nabla v|^2}{2} - fv \mathrm{d}x\]
		among all nonnegative functions $ v \in H^{1}_0(\Omega) $
		\item $ u $ is continuous and satisfies  
		\[
		\begin{cases}
			-\Delta u \geq f & \text{ in } \Omega \\
			u \geq 0 & \text{ in } \Omega\\
			-\Delta u = f & \text{ in } \{u > 0\} 
		\end{cases}
		\]
		in the weak sense.
	\end{enumerate}
\end{prop}
The main properties of solutions to \eqref{eq:obstacle_gamma} are summarized in the following theorem.
\begin{thm}[Proposition 3.9 in \cite{NRV20}]
	\label{thm:known}
	Consider a sequence $M_n \searrow 0$ and denote the corresponding solution to  \eqref{eq:obstacle_gamma} by $(u_n, \xi_n, \alpha_n)$. Then the following hold
\begin{enumerate}
	\item $ \alpha_n \searrow \alpha_0 $, with $\alpha_0$ given by \eqref{eq:alphazero},
	\item $ u_n \rightharpoonup 0 $ in $ W^{2,p}(\Gamma) $ and $ u_n \rightarrow 0 $ in $ C^{1,\gamma}(\Gamma) $ for all $0 \leq \gamma <1$ and $1\leq p <\infty$,
	\item For any $\delta > 0$ there exists $n_0$ such that if $ \mathrm{dist}(x, S) > \delta $, then $ u_n(x) = 0 $ for all $n\geq n_0$.
\end{enumerate}	
\end{thm}

We recall here a classical comparison principle for solutions to variational inequalities, see \cite[Theorem 2.6.4]{KS80}. Recall that $ G = (G(x)_{ij}) $ is uniformly elliptic in a domain $ \Omega $ if there exist $ \Lambda $, $ \lambda >0 $ such that
\[\Lambda |x|^{2} \geq x^{\top}G(x)x \geq \lambda |x|^{2} \faskip \forall x\in \Omega.\]
We call $ \frac{\Lambda}{\lambda} $ the ellipticity constant of $ G $.
\begin{lem}
	\label{lem:mp}
	Let $\Omega \subset \RR^{n}$ be a smooth bounded open set, $G: \Omega \rightarrow \RR^{n \times n}$ is a bounded, uniformly elliptic matrix, and $f\in C(\Omega)$. Let $u \in H^{1}(\Omega) \cap C(\Omega)$, $v \in H^{1}(\Omega)$ be nonnegative functions satisfying 
	\[
		- \divg(G\nabla u) = f\mathbb{1}_{\{u>0\}}, ~~ -\divg(G\nabla v)\geq f ~~ \text{in } \Omega,
	\]
	in the weak sense, and $v\geq u$ on $\partial \Omega$. Then, $v\geq u$ in $\Omega$. 
\end{lem}

The following is a restatement of the nondegeneracy property for the obstacle problem, see e.g. \cite[Lemma 2.3.1]{Fri82}. 
\begin{lem}[Nondegeneracy]
	\label{lem:nondeg}
	Let $G: B_1 \rightarrow \RR^{n \times n}$ be a bounded, uniformly elliptic matrix with $C^2$ coefficients. There exists some $\kappa_0 > 0$ depending only on the dimension and the ellipticity constant of $G$ such that the following holds: If $f \in C(\overline{B_1})$, $f \leq -\kappa_0$, and $u \in C(\overline{B_1})$ is a function satisfying 
	\[
	\begin{cases}
		-\divg(G\nabla u) = f\mathbb{1}_{\{u > 0\}} & \text{ in } B_1, \\
		0 \leq u \leq 1 & \text{ in } B_1,	
	\end{cases}
	\]
	in the weak sense, then $u \equiv 0$ in $B_{1/2}$.
\end{lem}

In order to prove the regularity of the free boundary in Theorem \ref{mainthm:morse}, it will suffice to prove that, for sufficiently small mass, the free boundary is a $ C^1 $ curve near maximum points, the higher regularity being given by classical methods, such as the hodograph-Legendre transform \cite[Section 2.2]{Fri82}. The next proposition is a consequence of the results in \cite{Caf77}, see also sections 2.4 and 2.5 in \cite{Fri82}.

\begin{prop}
	\label{prop:reg}
	Let $ G \in C^{3, \alpha}(B_1) $, $ f \in C^{0,\delta}(B_1) $, $\delta\in (0,1)$, such that $ f\leq -1 $ and $ \varepsilon < \frac{1}{8}$. There exist $ \rho_0 $ and $ \mu $, depending on $ \varepsilon $ and the ellipticity constant of $ G $ such that the following holds. If $ u \in C^{1,1}(B_1) $ is a solution to
	\[
		\begin{cases}
			-\divg(G\nabla u) \geq f & \text{ in } B_1, \\ 
			u\geq 0 		& \text{ in } B_1, \\
			-\divg(G\nabla u) = f	& \text{ in }\{u > 0\},
		\end{cases}	
	\]
	such that $ 0 \in \partial \{u > 0 \}$ and 
	\[\frac{|\{u = 0\} \cap B_{\rho}|}{|B_{\rho}|} \geq \varepsilon	\]
	for some $ 0 < \rho < \rho_0 $, then $ \partial \{u > 0\} \cap B_{\mu}$ is the graph of a $ C^1 $ function. 
\end{prop}

Throughout this article, $ C $ denotes a positive constant, whose value may change from line to line. Moreover, $ M $ (resp. $ M_n $) will denote the total mass of the solution $ u $ to \eqref{eq:obstacle_gamma} (resp. $ u_n $) while $ m $ will denote the mass in a neighborhood of a point. 

\section{Maxima described by a positive homogeneous function}
\label{sect:hom}

\subsection{Preliminary computations}

Throughout this section, we will let the assumptions of Theorem \ref{mainthm:isolated_hom} hold. Without loss of generality, we will assume that $ \Theta = B_R $, so that \eqref{eq:hyp_hom} becomes
\[g(x) = \gmax - f(x) + r(x) \faskip \forall x\in B_R.\]
Up to making $ R $ smaller, we can assume
\begin{equation}
\label{eq:hypothesis_nondeg}
	g(x) \leq \gmax - \frac{1}{2} f(x) \faskip \forall x\in B_R.
\end{equation}
Furthermore, we will assume that \eqref{eq:metric} holds in this coordinate system.

Throughout this section, $M_n$ will denote a sequence of positive masses such that $ M_n\rightarrow 0 $, and $(u_n, \alpha_n, \xi_n)$ will denote the solutions to \eqref{eq:obstacle_gamma} with mass $M_n$. 

In $B_{R}$, the right-hand side of \eqref{eq:obstacle_gamma} may be written as
\begin{align*}
	 (1+\alpha_n) g - 1  &=  (1 + \alpha_0 + (\alpha_n - \alpha_0))g - 1  \\
		      &=  \left(\frac{1}{\gmax} + (\alpha_n - \alpha_0)\right)\left( \gmax - f(x) + r(x) \right)  - 1  \\
		      &=   (\alpha_n - \alpha_0) \gmax - \left(\frac{1}{\gmax} + (\alpha_n - \alpha_0)\right) f(x) + \left(\frac{1}{\gmax} + (\alpha_n - \alpha_0)\right) r(x) .
\end{align*}
We will denote $\beta_n := \alpha_n - \alpha_0$. By Theorem \ref{thm:known} (1), we know that $\beta_n \searrow 0$.

Now, by Theorem \ref{thm:known} (3), by taking $ n $ large, we may assume $u_n \equiv 0$ on $ \partial B_R $. Define
\[m_n := \int_{B_R} u_n \mu(x)\mathrm{d}x ,\]
where $\mu(x) = \sqrt{\mathrm{det}(H(x))}$ is the Riemannian volume of $ \Gamma $. We will define the rescalings
\begin{equation}
\label{eq:nondeg_rescaling}
u_n = m_n^{(\gamma + 2)/(\gamma + 4)} \gmax^{(\gamma + 4)/\gamma}U_{n} (\gmax^{-2/\gamma} m_n^{-1/(\gamma + 4)} x),
\end{equation}
where $U_n$ is a function defined in $B_{m_n^{-1/6}R}$ and is extended to a function in $H^{1}(\RR^{2})$ by zero. Denote $ \varepsilon_n = m_n^{1/(\gamma + 4)} $ and $R_n := \varepsilon_n^{-1}R$ .

Computing
\begin{align*}
	m_n = \int_{B_R} u_n \mu(x) \mathrm{d}x &= m_n^{(\gamma + 2)/(\gamma + 4)}\int_{B_R} U_n(\gmax^{-2/\gamma} \varepsilon_n^{-1} x) \sqrt{\det (H(x))} \mathrm{d}x \\
					  &= m_n \gmax \int_{B_{R_n}} U_n(y) \sqrt{\det H(\varepsilon_n y)} \mathrm{d}y,
\end{align*}
we obtain $\| U_n\|_{L^{1}(\RR^{2},\mu_n)} = \gmax^{-1}$, where we denote $ \mu_n(y) = \sqrt{\det H(\gmax^{2/\gamma} \varepsilon_n y)} $. By \eqref{eq:metric}, we have that $ \mu_n \rightarrow 1$ uniformly in compact sets. 

Moreover, we can express the right hand side of \eqref{eq:obstacle_gamma} in the $ y = \varepsilon_n^{-1} x$ variable
\begin{align*}
	 (1+\alpha_n) g - 1  &= \beta_n \gmax - \left(\frac{1}{\gmax} + \beta_n\right) f(\gmax^{2/\gamma} \varepsilon_n y) + \left(\frac{1}{\gmax} + \beta_n\right) r(\gmax^{2/\gamma} \varepsilon_n y) \\
			     &= \gmax \varepsilon_n^{\gamma}\left( \frac{\beta_n}{\varepsilon_n^{\gamma}} - \left(1 + \frac{\beta_n}{\gmax}\right) f(y) + \left(\frac{1}{\gmax} + \beta_n\right) \frac{r(\varepsilon_n y)}{\gmax \varepsilon_n^{\gamma}}  \right). 
 \end{align*}
Note that, for $ |y| \leq T $, we have
\[
	\left| \frac{\beta_n}{\gmax} f(y) + \left(\frac{1}{\gmax} + \beta_n\right) \frac{r(\varepsilon_n y)}{\gmax\varepsilon_n^{\gamma}}\right| \leq \frac{\beta_n}{\gmax} \|f\|_{L^{\infty}(B_T)} + \frac{2T^{\gamma}}{\gmax^2}\frac{|r(\varepsilon_n y)|}{|\varepsilon_n y|^{\gamma}} \rightarrow 0 \text{ as }n\rightarrow \infty, 
\]
so that we may define
\[r_n (y)= \frac{\beta_n}{\gmax} f(y) + \left(\frac{1}{\gmax} + \beta_n\right) \frac{r(\varepsilon_n y)}{\gmax \varepsilon_n^{\gamma}}\]
and we have $r_n \rightarrow 0$ uniformly in compact sets. We will further set $ \gmax = 1 $ since the precise value of this parameter plays no role at this stage in our analysis.

Hence, the equation satisfied by $ U_n $ may be written as
\begin{equation}
	\label{eq:U}
	-\divg(H(\gmax^{2/\gamma}\varepsilon_n^{1/6} y)\nabla U_n(y)) = \left( \frac{\beta_n}{\varepsilon_n^{\gamma}} - f(y) + r_n(y) \right)\mathbb{1}_{\{U_n > 0\}}.
\end{equation}  

We will denote $H_n(y) = H(\gmax^{2/\gamma}\varepsilon_n y)$. By \eqref{eq:metric}, for any compact set $K \subset \RR^{2}$ and any function $ \varphi \in C^{2}(K) $, we have 
\begin{equation}
\label{eq:approx_diff}
	-\divg(H_n \nabla \varphi) = -\Delta \varphi + o(1).
\end{equation}
Finally, if $\lambda > 0$ is a lower bound for the ellipticity constant of $H$ in $B_R$, then $\lambda$ is also a lower bound for the ellipticity constant of $H_n$ in $B_{R_n}$.

\subsection{The limit problem}
\label{sect:limit_prob}

In view of \eqref{eq:U} and \eqref{eq:approx_diff}, if $ \frac{\beta_n}{\varepsilon_n^{\gamma}}\rightarrow \overline{\alpha} $ and $ (U_n)_{n\in\NN} $ converges in a strong sense, its limit should satisfy
\begin{equation}
\label{eq:obstacle_f}
	\begin{cases}
		-\Delta v = \left(\overline{\alpha} - f(y)  \right) \mathbb{1}_{\{v > 0\}}, \\
		v \geq 0, v\in L^{1}(\RR^{2}).
	\end{cases}
\end{equation}

We will first show that this problem has a unique solution.
\begin{prop}
	\label{prop:existence}
	Given $ \overline{\alpha} > 0 $, there exist a unique weak solution $ v \in H^{1}(\RR^{2}) $ to \eqref{eq:obstacle_f}.
\end{prop}
\begin{proof}
	First, let us prove uniqueness. Let $v_1$, $v_2$ be two solutions to \eqref{eq:obstacle}. The weak formulation of \eqref{eq:obstacle} reads	
\[
	\int_{\RR^{2}} \nabla v_i \cdot \nabla (\varphi - v_i) + (f - \overline{\alpha}) (\varphi - v_i) \mathrm{d}x \geq 0 \faskip \forall \varphi \in H_0^{1}(\RR^{2}), ~\varphi \geq 0.
\]
Testing $ v_1 $ and $ v_2 $ with each other yields
\begin{align*}
	\int_{\RR^{2}} |\nabla (v_1 - v_2)|^{2} &= -\left( \int_{\RR^{2}} \nabla v_1 \cdot \nabla (v_2 - v_1) + \nabla v_2 \cdot \nabla(v_1 - v_2) \right) \\
						&\leq - \left( \int_{\RR^{2}} (f - \overline{\alpha}) (v_2 - v_1) + (f - \overline{\alpha}) (v_1 - v_2) \right)  = 0,
\end{align*}
thus, $\nabla(v_1 - v_2) \equiv 0$, which means that $v_1 - v_2$ is constant. This implies that $v_1 \equiv v_2$ since both have compact support.

	For the existence, we will show that solving \eqref{eq:obstacle_f} is equivalent to solving a minimization problem of the form
\begin{equation}
\label{eq:min}
	\min \left\{ J(v) := \int_{\RR^{2}} \frac{|\nabla v|^{2}}{2} + f v \mathrm{d}y: v\in H^{1}(\RR^{2}),~ v\geq 0, ~ \int_{\RR^{2}}v = 1  \right\}.
\end{equation}

Let us show that \eqref{eq:min} has a minimizer. Indeed, fix any $ \varphi \in C^{1}_c(B_1) $, $ \varphi \geq 0 $, $ \varphi \equiv 1 $ in $ B_{1/2} $, and we have that
\[w = \frac{1}{\int_{\RR^{2}}\varphi \mathrm{d}y} \varphi\]
	is an admissible test function. 

	Take a minimizing sequence $ (v_n)_{n\in\NN} $. Therefore, we have that $ J(v_n) \geq C $. In particular, $ \|\nabla v_n \|_{L^{2} (\RR^{2})} \leq C $. Moreover, by Nash's inequality, we have
	\[\|v_n\|_{L^{2}(\RR^{2})}^{4} \leq C \|\nabla v_n \|_{L^{2}(\RR^{2})} \| v_n \|_{L^{1}(\RR^{2})} \leq C.\]
	Thus, $ (v_n)_{n\in \NN} $ is a bounded sequence in $ H^{1}(\RR^{2}) $. By the Banach-Alaoglu theorem, there exists some $ \overline{v} \in H^{1}(\RR^{2}) $ and a subsequence of $ (v_n) $ such that $ v_n \rightharpoonup \overline{v} $ weakly in $ H^{1} (\RR^2)$. In each smooth compact set we can apply the Rellich-Kondrachov theorem to extract a further subsequence such that $ v_n \rightarrow \overline{v} $ in $ L^{2}_{\loc}(\RR^{2}) $ and a.e. in $ \RR^{2} $.

	We want to show that $ \overline{v} $ is admissible. Clearly, $ \overline{v} \geq 0 $, so it remains to prove that $ \int_{\RR^{2}} \overline{v} \mathrm{d}y = 1  $. To verify this, note that $ \nu_n =  v_n \mathrm{d}x $ are probability measures on $ \RR^{2} $. Moreover, the family $ (\nu_n)_{\{n\in \NN\}} $ is tight For any $ R > 0 $, we have
	\begin{align*}
		\int_{\RR^{2}} f v_n \mathrm{d}y &= \int_{\RR^{2}\setminus B_R} f v_n \mathrm{d}y + \int_{B_R} f v_n \mathrm{d}y \\
						    &\geq \int_{\RR^{2}\setminus B_R} |y|^{\gamma} f \left( \frac{y}{|y|} \right)  v_n \mathrm{d}y \\
						    &\geq R^{\gamma} \min_{\partial B_1} f \int_{\RR^{2}\setminus B_R} v_n \mathrm{d}y \\
						    &= R^{\gamma} \min_{\partial B_1}f \nu_n(\RR^{2}\setminus	B_R).
	\end{align*}
	Recalling that $ \int_{\RR^{2}} f(y) v_n \mathrm{d}y \leq J(v_n) \leq C $ we conclude
	\[\nu_n (\RR^{2}\setminus B_R) \leq \frac{C}{R^{\gamma} \min_{\partial B_1}f},\]
	which implies that $ (\nu_n) $ is tight. Prokhorov's theorem \cite[Theorem 8.6.1]{Bog07} implies that there exists a probability measure $ \nu $ such that $ \nu_n \rightharpoonup \nu $ in the weak sense of measures, up to a subsequence. But by testing with smooth functions with compact support, we readily verify that $ \nu = \overline{v} \mathrm{d}x$. This implies that the whole sequence $ (\nu_n) $ converges to $ \nu $ and that $ \|\overline{v} \|_{L^{1}(\RR^{2})} = 1$, and therefore $ \overline{v}  $ is admissible.

	The weak lower semicontinuity of $ u \mapsto \| \nabla u \|^{2}_{L^{2}(\RR^{2})} $ in $ H^{1}(\RR^{2}) $ yields
	\[\liminf_{n\rightarrow \infty}\int_{\RR^{2}} \frac{|\nabla v_n |^{2}}{2} \mathrm{d}y \geq \int_{\RR^{2}} \frac{|\nabla \overline{v}|^{2}}{2} \mathrm{d} y.  \]
	On the other hand, Fatou's lemma yields
	\[\liminf_{n\rightarrow \infty}\int_{\RR^{2}} f v_n \mathrm{d}y \geq \int_{\RR^{2}} f \overline{v} \mathrm{d}y.  \]
We conclude that $ \overline{v} $ is a minimizer of \eqref{eq:min}.

Now, we will show that the Euler-Lagrange equation of \eqref{eq:min} is precisely \eqref{eq:obstacle_f}. Take $\varphi \in C_c^1(\RR^2)$ a nontrivial nonnegative function, and $\delta >0$. Define
\[v_{\delta} := \frac{v + \delta \varphi}{\int_{\RR^2} v+ \delta \varphi \mathrm{d}y} = \frac{1}{1 + \delta \int_{\RR^2} \varphi \mathrm{d}y} (v+\delta \varphi),\]
which is admissible for \eqref{eq:min}. By the fact that $u$ is a minimizer of \eqref{eq:min}, we have $J(u_{\delta}) \geq J(u)$. On the other hand,
\begin{align*}
J(v_{\delta}) - J(v) =& \int_{\RR^2} \frac12\left( \frac{1}{\left( 1+ \delta \int_{\RR^2} \varphi \mathrm{d}y\right)^2}|\nabla v_{\delta}|^2 - |\nabla v|^2 \right) + f\left(\frac{1}{1 + \delta \int_{\RR^2} \varphi \mathrm{d}y} (v+\delta \varphi) - v\right)\mathrm{d}y \\
=& \int_{\RR^2} \frac12\left( \left( \frac{1}{\left( 1 + \delta \int_{\RR^2} \varphi \mathrm{d}y\right)^2}- 1\right) |\nabla v|^2 +  \frac{\delta}{\left( 1+ \delta \int_{\RR^2} \varphi \mathrm{d}y\right)^2} (2 \nabla v \cdot \nabla \varphi + \delta |\nabla \varphi|^2) \right) \\ & + f\left(\left(\frac{1}{1 + \delta \int_{\RR^2} \varphi \mathrm{d}y} - 1 \right)v + \frac{\delta}{1 + \delta \int_{\RR^2} \varphi \mathrm{d}y} \varphi\right)~ \mathrm{d}y. 
\end{align*}

Recalling that
\begin{align*}
\frac{1}{1 + \delta \int_{\RR} \varphi \mathrm{d}y} &= 1 - \delta \int_{\RR} \varphi \mathrm{d}y + o(\delta), \\
\frac{1}{\left(1 + \delta \int_{\RR} \varphi \mathrm{d}y\right)^2} &= 1 - 2\delta \int_{\RR^2} \varphi \mathrm{d}y + o(\delta),
\end{align*}
as $\delta \rightarrow 0$, we get
\begin{align*}
J(v_{\delta})-J(v)= \delta \int_{\RR^2} \nabla v \cdot \nabla\varphi  + f\varphi \mathrm{d}y - \delta \left(\int_{\RR^2} \varphi \mathrm{d}y	\right) \left(\int_{\RR^2} |\nabla v|^2 - f v \mathrm{d}y\right)  + o(\delta).
\end{align*}

Dividing by $\delta$ and letting $\delta \searrow 0$, we have
\[\int_{\RR^2} \nabla v \cdot \nabla\varphi  + f \varphi \mathrm{d}y \geq \left(\int_{\RR^2} \varphi \mathrm{d}y\right) \left(\int_{\RR^2} |\nabla v|^2 - f \mathrm{d}y\right),\]
which can be rewritten, defining
\begin{equation}
	\label{eq:alpha}
	\overline{\alpha} := \int_{\RR^2} |\nabla v|^2 - fv \mathrm{d}x,
\end{equation}
as 
\[
\int_{\RR^2} \nabla v \cdot \nabla\varphi  +( f - \overline{\alpha}) \varphi \mathrm{d}y \geq 0 \faskip \forall \varphi \in C^2_{c}(\RR^2),~\varphi \geq 0,
\]
which is precisely the weak formulation of \eqref{eq:obstacle_f}. 

Finally, define
\begin{equation}
	\label{eq:v_rescale}
	v_{\overline{\alpha}}(y) = \overline{\alpha}^{1 + 2/\gamma} \overline{v}\left( \frac{y}{\overline{\alpha}^{1/\gamma}} \right),  
\end{equation}
and $ v_{\overline{\alpha}} $ solves \eqref{eq:obstacle_f}.
\end{proof}

Using classical regularity estimates for elliptic equations see \cite[Theorem 9.11]{GT83}, we see that the solution to \eqref{eq:obstacle_f} is in $ W^{2,p}_{\loc}(\RR^{2}) $ for any $ p\in [1,\infty)$, since the right hand side of \eqref{eq:obstacle_f} is locally bounded. In particular, $ v $ is continuous.

\begin{prop}
	\label{prop:compact}
	Let $v$ be the unique weak solution to \eqref{eq:obstacle_f}. Then, $v$ has compact support.
\end{prop}

\begin{proof}
	Let $\kappa_0 > 0$ be the constant from Lemma \ref{lem:nondeg}, and let $R > 1$ be such that 
	\begin{equation}
	\label{eq:rhs}
	\overline{\alpha} - f(x) \leq -\|v\|_{L^{1}(\RR)} \kappa_0 \faskip \forall x \in \RR^{2}\setminus B_{R/2},
	\end{equation}
	which exists by homogeneity and the fact that $ f(x) > 0 $ for every nonzero $ x $. In particular, $v$ is subharmonic in $ \RR^{2}\setminus B_{R/2} $, whence
	\[ \sup_{B_1(x_0)} v \leq \int_{B_2(x_0)} v(x) \mathrm{d}x\leq \|v\|_{L^{1}(\RR^{2})} \faskip \forall x \in \RR^{2}\setminus B_{R+1}, \]
by the mean value property and the fact that $v\geq 0$. 

	For any $x_0\in \RR^{2}\setminus B_{R+1}$, we can apply Lemma \ref{lem:nondeg} to $ v(x) = \frac{v(x_0 + x)}{\|v\|_{L^{1}(\RR^{2})}} $ to get that $v \equiv 0$ in $B_{1/2}(x_0)$. Thus, the support of $v$ is contained in $B_{R+1}$.
\end{proof}

\subsection{Convergence}
\label{sect:hom_subseq}

In this section we prove the convergence of $ U_n $ as defined in \eqref{eq:nondeg_rescaling}, first up to a subsequence, then along the whole parameter $ m\rightarrow 0 $.

\begin{prop}
	\label{prop:nondeg_q_bound}
	$ \limsup_{n\rightarrow \infty} \frac{\beta_n}{\varepsilon_n^{\gamma}} < \infty$
\end{prop}

\begin{proof}
	We argue by contradiction. Assume that, for some subsequence, $ \frac{\beta_n}{\varepsilon_n^{\gamma}} \rightarrow \infty $ as $n\rightarrow \infty$. Therefore, there exists $ \overline{r}> 0 $ independent of $ n $ such that 
	\[\frac{\beta_n}{\varepsilon^{\gamma}} - f(y) + r_n(y) \geq \frac{\beta_n}{2\varepsilon_n^{\gamma}} \faskip \forall y\in B_{\overline{r}}.\]
	
	Define the function
	\[
		\varphi(y) = \frac{\beta_n}{16 \varepsilon_n^{\gamma}} \left( \overline{r}^{2} - |y|^{2} \right) 
	\]
	which satisfies
	\[
		-\divg(H_n \nabla \varphi) = \frac{\beta_n}{4 \varepsilon_n^{\gamma}} + o(1).
	\]
	
	For $n$ large enough, $-\divg(H_n \nabla U_n) \geq -\divg(H_n \nabla \varphi)$ in $B_r$, and since $U_n \geq 0$ and $\varphi = 0$ on $ \partial B_{\overline{r}} $, using the comparison principle we conclude $U_n \geq \varphi$ in $ B_{\overline{r}} $. In particular,
	\[U_n \geq \frac{3 \overline{r}^{2}\beta_n}{64 \varepsilon_n^{\gamma}} \text{ in }B_{\overline{r}/2}.\]
Therefore,
\[1 = \int_{\RR^{2}} U_n \mu_n(y)\mathrm{d}y \geq \int_{B_{\overline{r}/2}} U_n \mu_n(y) \mathrm{d}y \geq \frac{3 \overline{r}^{2}\beta_n}{64 \varepsilon_n^{\gamma}} \mu_n(B_{\overline{r}/2}) \rightarrow \infty, \]
a contradiction. We conclude that $\frac{\beta_n}{\varepsilon_n^{\gamma}}$ must be bounded.
\end{proof}

\begin{prop}
	\label{prop:nondeg_infty_bd}
	There exist $C_0, R_1 > 0$ such that $ \supp(U_n) \subset B_{R_1} $ and 
	\begin{equation}
	\label{eq:nondeg_infty_bd}
	\|U_n\|_{L^{\infty}(\RR^{2})} \leq C_0 \frac{\beta_n}{\varepsilon_n^{\gamma}} ,
	\end{equation}
	for every $ n\in \NN $.
\end{prop}
\begin{proof}
	Take $ v $ the solution to \eqref{eq:obstacle_f} with $ \overline{\alpha} = 1 $, and define the function 
	\[\varphi(y) = 4^{2/\gamma} \left( \frac{\beta_n}{\varepsilon_n^{\gamma}} \right)^{(\gamma + 2)/\gamma} \Phi \left( \left( \frac{4 \beta_n}{\varepsilon_n^{\gamma}} \right)^{-1/\gamma} y  \right). \]
	
	Then, $ \varphi \geq 0 $ and a straightforward computation shows that $ \varphi $ satisfies
	\begin{align*}
		-\Delta \varphi = \left( \frac{\beta_n}{\varepsilon_n^{\gamma}} - \frac{1}{4} f(y) \right) \mathbb{1}_{\{\varphi>0\}}.
	\end{align*}
By Proposition \ref{prop:compact}, there is $ R_0 > 0 $ such that $ \supp(v)\subset B_{R_0} $. By definition, $ \supp(\varphi) \subset \left( \frac{4\beta_n}{\varepsilon_n^{\gamma}} \right)^{1/\gamma}B_{R_0} $, and by Proposition \ref{prop:nondeg_q_bound}, we may choose $C > 0$ such that $ \supp(\varphi) \subset B_{CR_0} $ for any $ n $. By \eqref{eq:approx_diff}, we know that 
	\begin{align*}
		-\divg(H_n \nabla \varphi) &\geq \frac{\beta_n}{\varepsilon_n^{\gamma}} - \frac{1}{4} f(y) +o(1) \\
					   &\geq \frac{\beta_n}{\varepsilon_n^{\gamma}} - \frac{1}{2} f(y)
	\end{align*}
	in $ B_{CR_0} $, for $ n $ large. In $ B_{R_n}\setminus B_{CR_0} $, we have $ -\divg(H_n \varphi) =0 \geq  \frac{\beta_n}{\varepsilon_n^{\gamma}} - \frac{1}{2} f(y)$.

	Recall that $ \supp(U_n)\subset B_{R_n} $ for $ R_n = \varepsilon_n^{-1} R $, whence we have $ U_n = \varphi = 0$ on $ \partial B_{R_n} $, so that by Lemma \ref{lem:mp} we have $ \varphi \geq U_n $. In particular, $ U_n \equiv 0 $ outside $ B_{CR_0} $.
	
	To prove \eqref{eq:nondeg_rescaling}, recall that 
	\[-\divg(H_n \nabla U_n) \leq \frac{\beta_n}{\varepsilon_n^{\gamma}} ~ \text{ in } B_{CR_0}.\]
	Thus, taking 
	\[\psi(y) = \frac{\beta_n}{\varepsilon_n^{\gamma}}((CR_0)^{2} - |x|^{2})\]
	we have $ \psi \geq 0 $ and $ -\divg(H_n \nabla \psi) = 2 \frac{\beta_n}{\varepsilon_{n}^{\gamma}} + o(1)\geq \frac{\beta_n}{\varepsilon_n^{\gamma}} $ in $ B_{CR_0} $ for $ n $ large, and $ \psi =U_n =0 $ on $ \partial B_{CR_0} $. By Lemma \ref{lem:mp}, $ \psi \geq U_n $, therefore 
	\[\|U_n\|_{L^{\infty}(\RR^{2})} \leq \frac{(CR_0)^{2}\beta_n}{2\varepsilon_n^{\gamma}},\]
	as claimed. 
\end{proof}

\begin{cor}
	\label{cor:nondeg_inf_bound}
	$ \liminf_{n \rightarrow \infty} \frac{\beta_n}{\varepsilon_n^{\gamma}} > 0$.
\end{cor}
\begin{proof}
	By Proposition \ref{prop:nondeg_infty_bd}, we have 
	\[
		1 = \int_{B_{r_1}} U_n \mu_n(y)\mathrm{d}y \leq 2r_1^{2} \mu_n(B_{r_1})  \frac{\beta_n}{\varepsilon_n^{\gamma}},
	\]
which implies the claim by recalling that $\mu_n(B_{R_1}) \rightarrow |B_{R_1}| > 0$.
\end{proof}

We now turn to the convergence of the interfaces. First, take $ v_{\alpha} $ as defined in \eqref{eq:v_rescale}. It is then easy to see that the function $ \alpha \mapsto v_{\alpha} $ is continuous for the topology of uniform convergence.

The Hausdorff convergence of the interfaces will follow from the following proposition.
\begin{prop}
	\label{prop:hom_sandwich}
	Assume $ \frac{\beta_n}{\varepsilon_n^{2}} \rightarrow \overline{\alpha} $. For every $\delta \in (0, \overline{\alpha})$, there exists $ n_0 \in \NN $ such that 
	\begin{equation}
	\label{eq:hom_sandwich}
		v_{\overline{\alpha} - \delta} \leq U_n \leq v_{\overline{\alpha} + \delta} \text{ in } \RR^{2}.
	\end{equation}
	for every $n\geq n_0$.
\end{prop}
\begin{proof}
	Fix $\delta > 0$ and $ r_2 > 0 $ such that $\supp (v_{\overline{\alpha} + \delta}) \cup B_{r_1} \subset  B_{r_2} $. There exists $n_1 \in \NN$ such that 
	\begin{equation}
	\label{eq:hom_rhs}
	\overline{\alpha} - \frac{\delta}{2} - f(y) \leq \frac{\beta_n}{\varepsilon_n^{4}} - f(y) + o(1) \leq \overline{\alpha} + \frac{\delta}{2} - f(y) \faskip \forall y \in B_{r_2}.
	\end{equation}
	
	By the fact that $v_{\overline{\alpha} + \delta} \in C^{1,1}_c(B_{r_2}) $, for some $ n_1 \in \NN $ we have
	\begin{align*}
		-\divg(H_n \nabla v_{\overline{\alpha} + \delta}) &\geq \overline{\alpha} + \delta - f(y) + o(1) \\
							  &\geq \overline{\alpha} + \frac{\delta}{2} - f(y)
	\end{align*}
	in $ B_{r_2} $ for every $ n\geq n_2 $. Since $ v_{\overline{\alpha} + \delta} = U_n = 0 $ on $ \partial B_{r_2} $, we can apply Lemma \ref{lem:mp} to get $ v_{\overline{\alpha} + \delta} \geq U_n $ in $ B_{r_2} $. The upper bound in \eqref{eq:hom_sandwich} follows by the fact that $ v_{\overline{\alpha} + \delta} = U_n = 0$ outside $ B_{r_2} $.

	In $ \{v_{\overline{\alpha} - \delta} > 0\} $, we have that 
	\begin{align*}
		-\divg(H_n \nabla v_{\overline{\alpha} - \delta}) &= \overline{\alpha} - \delta -f(y) \\
								  &\leq \overline{\alpha} - \frac{\delta}{2} - f(y) \\
								  &\leq -\divg(H_n\nabla U_n),
	\end{align*}
	for $ n \geq n_3 $ for some $ n_3 \in \NN $. By the comparison principle we get $ U_n \geq v_{\overline{\alpha} - \delta} $ in $ \{v_{\overline{\alpha} - \delta} > 0\}$. The lower bound in \eqref{eq:hom_sandwich} follows for $ n_0\geq \max\{ n_1, n_2, n_3\} $, and the fact that $ U_n \geq 0 $ in $ \RR^{2} $. 
\end{proof}

\begin{proof}[Proof of Theorem \ref{mainthm:isolated_hom}]
	By Proposition \ref{prop:nondeg_q_bound} and Corollary \ref{cor:nondeg_inf_bound}, we may choose a subsequence such that $ \frac{\beta_n}{\varepsilon_n^{\gamma}} \rightarrow \overline{\alpha} $. 

	Classical regularity estimates (see \cite[Lemma 9.17]{GT83}) imply 
	\begin{equation}
		\label{eq:reg_estimate}
	\|U_n\|_{W^{2,p}(B_{r_1})} \leq C\left\|\frac{\beta_n}{\varepsilon_n^{\gamma}} - f(y) + r_n(y) \right\|_{L^{p}(B_{r_1})} \leq C \left\|\frac{\beta_n}{\varepsilon_n^{\gamma}} - f(y) + r_n(y) \right\|_{L^{\infty}(B_{r_1})} \faskip \forall p \in (1,\infty),
	\end{equation}
which implies due to Proposition \ref{prop:nondeg_q_bound}, that $U_n$ is a bounded sequence in $ W^{2,p}(B_{r_1}) $. For $ p=2 $, the Banach-Alaoglu theorem implies that we may extract a subsequence such that $U_n \rightharpoonup \overline{U} $ weakly in $ H^2(B_{r_1}) $. Moreover, up to a further subsequence, the Rellich-Kondrachov theorem implies that $U_n\rightarrow \overline{U}$ strongly in $ H^1(B_{r_1}) $ and a.e. in $ B_{r_1} $.

	The weak formulation of the equation satisfied by $U_n$ reads
	\[
		\int_{B_{r_1}} H_n \nabla U_n \cdot \nabla (\varphi - U_n) \mathrm{d}y \geq \int_{B_{r_1}} \left( \frac{\beta_n}{\varepsilon_n^{\gamma}} - f + r_n \right) (\varphi - U_n) \mathrm{d}y, \faskip \forall \varphi \in H^{1}_0(B_{r_1}),~ \varphi \geq 0.
	\]
	By the strong $ H^{1} $ convergence of $U_n$, we can take the limit in this inequality to get 
	\[
		\int_{B_{r_1}} \nabla \overline{U} \cdot \nabla (\varphi - \overline{U}) \geq\int_{B_{r_1}} \left( \overline{\alpha} - f  \right) (\varphi - \overline{U}) \mathrm{d}y, \faskip \forall \varphi \in H^{1}_0(B_{r_1}),~ \varphi \geq 0. 
	\]
	Since $\overline{U}\in W^{2,2}(B_{r_1})$, this equation is satisfied strongly, i.e. \eqref{eq:obstacle_f} holds 

	Recalling that $U_n \leq C \frac{\beta_n}{\varepsilon_n^{\gamma}} \mathbb{1}_{B_{r_1}}$, the dominated convergence theorem implies
	\[ \int_{B_{r_1}}U_n \mu_n(y) \mathrm{d}y\rightarrow \int_{B_{r_1}} \overline{U}\mathrm{d}y, \]
	therefore $ \|\overline{U}\|_{L^{1}(B_{r_1})} = 1 $.

	Now, the above discussion implies that $ \overline{U} $ is a minimizer for \eqref{eq:min}. That is, $ \overline{\alpha} $ must satisfy \eqref{eq:alpha}. This implies that any convergent sequences $ (\frac{\beta_n}{\varepsilon_n^{\gamma}}) $, $ (U_n) $ must converge to the same limit. This implies the convergence along the parameter $ M $.

	To finish, we prove the Hausdorff convergence of the interfaces. Let $ S = \partial\{\overline{U} > 0\} = \partial\{v_{\overline{\alpha}} > 0\} $. By \eqref{eq:v_rescale}, we have that 
	\[
		\partial\{v_{\overline{\alpha} \pm \delta} > 0\} = \left( 1 \pm \frac{\delta}{\overline{\alpha}} \right)^{1/\gamma} S.
	\]
	Then, $ \mathrm{dist}_H(\{v_{\overline{\alpha} + \delta} = 0\}, \{v_{\overline{\alpha} - \delta} > 0\}) \rightarrow 0 $ as $ \delta \rightarrow 0 $. 

	By Proposition \ref{prop:hom_sandwich}, for any $ \delta > 0 $ we have $ \partial \{U_n > 0\} \subset \{v_{\overline{\alpha} + \delta} > 0\} \setminus \{v_{\overline{\alpha} - \delta} = 0\}$ for $ n $ sufficiently large, which by the above discussion implies the desired convergence.
\end{proof}

Note that \eqref{eq:reg_estimate} and the Banach-Alaoglu theorem yield that $U_n \rightharpoonup \overline{U}$ in $W^{2,p}(B_{r_1})$ for any $p \in (1, \infty)$, in particular strongly in $C^{1,\gamma}(B_{r_1})$ for any $\gamma \in (0,1)$ by standard Sobolev embeddings.

\section{Morse functions}
\label{ssect:quadratic}

In this section, we treat in detail the case where $ g $ is a Morse function. For ease of reference, we record the formulas for the relative mass of each maximum point in the following theorem.

\begin{thm}
	\label{thm:lambda_i}
	With the notation and hypotheses of Theorem \ref{mainthm:morse}, if $ \sigma^i_{1,2} > 0 $ are the eigenvalues of $- D^2g(p_i) $, assume without loss of generality that $ \sigma^i_2 > \sigma_1^i $ and denote
	\[
		\begin{array}{cc}
			s_i = \sigma^i_2 - \sigma^i_1.
		\end{array}
	\]
Define
\begin{equation}
\label{eq:C0}
	C_0(s_i) = \left( \frac{3}{1-2s_i} \right)^{2} \left( \frac{1}{8} - \frac{1}{2}\sqrt{\frac{1}{16} - \frac{1}{2} \left( \frac{1-2s_i}{3} \right)^{2}} \right), 
\end{equation}
\begin{equation}
	\label{eq:beta}
	\beta(s_i) = \frac{C_0(s_i) (1-2s_i)}{3},
\end{equation}
\begin{equation}
	\label{eq:mass}
	M(s_i) = \frac{73\pi}{96} \frac{C_0(s_i)^2}{\sqrt{\frac{1}{16} - \beta(s_i)^2}}.
\end{equation}
Then,
\begin{equation}
	\label{eq:lambda_i}
	\lambda_i = \frac{M(s_i) \mathrm{tr}(-D^2 g(p_i))^{-3/2}}{\sum_{j=1}^{N} M(s_j) \mathrm{tr}(-D^2 g(p_j))^{-3/2}}. 
\end{equation}
\end{thm}

Recall that, if $ g $ is a Morse function, then around each maximum $ p $ we may write
\[g(x) = \gmax - x^{\top}A x + O(|x|^{3}),\]
in a suitable coordinate system, and with $ A $ a positive definite symmetric matrix.  Of course, this $ g $ satisfies the hypotheses of Theorem \ref{mainthm:isolated_hom} with $ \gamma = 2 $, so that the discussion in the preceding subsections remains valid.

In this case, the limit problem reads
\begin{equation}
	\label{eq:obstacle}
	\begin{cases}
		-\Delta v = (\overline{\alpha} - x^{\top} A x) \mathbb{1}_{\{v>0\}} & \text{ in } \RR^2, \\
		v \geq 0 & \text{ in } \RR^2,\\
		v \in L^{1}(\RR^{2}).
	\end{cases}
\end{equation}

For the subsequent analysis, we note that the number of parameters can be reduced in the following way. If $\Phi$ is a solution to 
\begin{equation}
	\label{eq:obst_reduced}
	\begin{cases}
	-\Delta \Phi = \left( 1 - y^{\top}\Tilde{A}y  \right)\mathbb{1}_{\{\Phi>0\}} & \text{in } \RR^{2},\\ 
		\Phi \geq 0 & \text{in } \RR^{2}, \\
		\Phi \in L^{1}(\RR^{2}), 
	\end{cases}
\end{equation}
where $\Tilde{A} = \frac{1}{\text{tr}(A)}A$, then,
\begin{equation}
	\label{eq:v_hom}
	v(x) = \frac{\overline{\alpha}^{2}}{\text{tr}(A)} \Phi \left( \sqrt{\frac{\text{tr}(A)}{\overline{\alpha}}} x \right) 
\end{equation}
solves \eqref{eq:obstacle}. Hence, up to an orthogonal change of coordinates, it suffices to study the case
\[
	A = \begin{pmatrix}
		s & 0 \\ 
		0 & 1-s
	\end{pmatrix}
\]
for $s \in (0, \frac{1}{2}]$.

\subsection{Radial case}
\label{ssect:radial_case}

In what follows we will consider the case in which $A$ is of the form
\[A = \frac{1}{2} I.\]
In this case, the solution $\Phi$ will be of the form $\Phi(x) = \psi(|x|)$ for some $\psi:\RR \rightarrow \RR$. Since $\psi(0) > 0$, for some interval $[0, \overline{r})$, we must have
\begin{align*}
	-\Delta \Phi = -\frac{1}{r} (r \psi'(r))' &= 1 -  \frac{r^{2}}{2},
\end{align*}
so that 
\[\psi(r) = \psi(0) -\frac{r^{2}}{4} + \frac{r^{4}}{32} \faskip \forall r \in [0, \overline{r}).\]

At $\overline{r}$, we must have $\psi(\overline{r}) = \psi'(\overline{r}) =0$. Thus
\begin{align*}
	\psi(\overline{r}) &= \psi(0) - \frac{\overline{r}^{2}}{4} +  \frac{\overline{r}^4}{32} = 0, \\
	\psi'(\overline{r}) &= -\frac{\overline{r}}{2} + \frac{\overline{r}^{3}}{8} = 0.
\end{align*}
Hence, $\overline{r} = 2$ and $\psi(0) = \frac{1}{2}$.

We conclude that the solution in this case is 
\begin{equation}
\label{eq:solution_radial}
\Phi(x) = \begin{cases}
	\frac{1}{2} - \frac{|x|^{2}}{4} + \frac{|x|^{4}}{32} & \text{if } |x| \leq 2\\
	0 & \text{ otherwise}
\end{cases}
\end{equation} 

\subsection{General case}
\label{ssect:general_case}

Recall that, in this case, 
\begin{align*}
A = \begin{pmatrix}
	s & 0 \\
	0 & 1-s
\end{pmatrix}
\end{align*}
for some $s \in (0, \frac{1}{2})$.

In this case, $\Phi$ should have the form
\begin{align*}
\Phi(x) = \frac{s x_1^4 + (1-s) x_2^4}{12} - \frac{|x|^2}{4} + \varphi(x),
\end{align*}
where $\varphi$ is harmonic inside the positivity set of $\Phi$. We will denote $p_2(x) = \frac{|x|^2}{4}$ and $p_4(x) = \frac{s x_1^4 + (1-s) x_2^4}{12}$. For $\Phi$ to be a solution to \eqref{eq:obstacle}, $\varphi$ must satisfy
\begin{equation}
\label{eq:bc_zero}
\varphi = p_2 - p_4 ~~ \text{on }\partial\{\Phi>0\},
\end{equation}
\begin{equation}
\label{eq:bc_one}
\nabla\varphi = \nabla p_2 - \nabla p_4 ~~ \text{on } \partial\{\Phi>0\}.
\end{equation}

We look for $\varphi$ of the form
\begin{align*}
\varphi(x) = C_0 + Y_2(x) + Y_4(x),
\end{align*}
where $Y_i$ is a harmonic polynomial of degree $i$.

By taking $x\cdot$\eqref{eq:bc_one}, we get, recalling that $ x\cdot \nabla P_m = mP_m$ for any homogeneous polynomial $P_m$ of degree $m$,
\begin{align*}
	x\cdot \nabla \varphi =x\cdot( \nabla Y_2 + \nabla Y_4 ) = 2Y_2 + 4Y_4 &= x\cdot( \nabla p_2 - \nabla p_4) = 2 p_2 - 4 p_4, \\
\end{align*}
hence,
\[
	 p_2 - Y_2 = 2 (Y_4 + p_4).
\]

Note that \eqref{eq:bc_zero} can be rewritten as
\begin{align*}
p_2- Y_2 = C_0 + Y_4 + p_4,
\end{align*}
so that 
\begin{equation}
	\label{eq:curves}
\begin{cases}
	Y_4 + p_4 = C_0 & \text{on } \partial\{\Phi>0\},\\
	p_2 - Y_2 = 2 C_0  & \text{on } \partial\{\Phi>0\}.
\end{cases}
\end{equation}

Writing 
\begin{align*}
	p_2 - Y_2 &= b_1 x_1^2 + b_2 x_2^2, \\
	p_4 + Y_4 &= a_1 x_1^4 + a_2 x_2^4 + a_3 x_1^2x_2^2,
\end{align*}
and using polar coordinates $x_1 = \rho \cos(\theta)$ and $x_2 = \rho \sin(\theta)$, we can write \eqref{eq:curves} as 
\begin{align*}
	\rho^2 (b_1 \cos(\theta)^2 + b_2 \sin(\theta)^2) &= 2C_0, \\
	\rho^4 (a_1 \cos(\theta)^4 + a_2 \sin(\theta)^4 + a_3 \cos(\theta)^2 \sin(\theta)^2) &= C_0,
\end{align*}
which can be rewritten as
\begin{align*}
	\rho^4 = \frac{C_0}{a_1 \cos(\theta)^4 + a_2 \sin(\theta)^4 + a_3 \cos(\theta)^2 \sin(\theta)^2} &= \left( \frac{2C_0}{b_1 \cos(\theta)^2 + b_2 \sin(\theta)^2} \right)^2,
\end{align*}
and therefore
\[
	 \frac{b_1^2 \cos(\theta)^4 + b_2^2 \sin(\theta)^4 + 2b_1b_2 \cos(\theta)^2 \sin(\theta)^2}{a_1 \cos(\theta)^4 + a_2 \sin(\theta)^4 + a_3 \cos(\theta)^2 \sin(\theta)^2} = 4C_0.
\]
Since this identity holds for every $ \theta $, we must have
\begin{align*}
	\begin{array}{ccc}
		b_1^2 = 4C_0a_1, &
		b_2^2 = 4C_0a_2, &
	2b_1b_2 = 4C_0a_3.
	\end{array}
\end{align*}

We now turn to expressing $\varphi$ in the base of harmonic polynomials. Note that, by the invariance of equation \eqref{eq:obst_reduced} under the transformations $(y_1, y_2) \mapsto (-y_1 ,y_2) $ and reflections around the origin, we consider only homogeneous polynomials which are even in each variable, i.e.
\[Y_2 (x) = \beta_2(x_1^2 -x_2^2),~~ Y_4(x) = \beta_4(x_1^4 + x_2^4 -6x_1^2x_2^2),\]
where $ \beta_2 $ and $ \beta_4 $ are real numbers to be determined. Hence 
\[
	\begin{array}{ccccc}	
		b_1 = \dfrac{1}{4} - \beta_2, & b_2 = \dfrac{1}{4} + \beta_2, & a_1 = \dfrac{s}{12} + \beta_4, & a_2 = \dfrac{1-s}{12} + \beta_4, & a_3 = -6\beta_4.
	\end{array}
\]
Plugging this in the equations for $a_i$, $b_i$, we get
\begin{subequations}
\label{eq:beta_system}
\begin{align}
	\left( \frac{1}{4} - \beta_2 \right)^2 &= 4C_0\left( \frac{s}{12} + \beta_4\right), \label{eq:beta_square1}\\
	\left( \frac{1}{4} + \beta_2 \right)^2 &= 4C_0\left( \frac{1-s}{12} + \beta_4\right),  \label{eq:beta_square2}\\
	\left(\frac{1}{4} - \beta_2 \right) \left(\frac{1}{4} + \beta_2 \right) &= -12C_0\beta_4 \label{eq:beta_product}.
\end{align}
\end{subequations}

Subtracting \eqref{eq:beta_square1} from \eqref{eq:beta_square2} yields 
\begin{equation}
	\label{eq:beta2}
	\beta_2 =  \frac{C_0 \left(1-2s\right)}{3} . 
\end{equation}
On the other hand, adding \eqref{eq:beta_product} and \eqref{eq:beta_square2} and plugging \eqref{eq:beta2} gives
\begin{align*}
	\frac{1}{8} + \frac{\beta_2}{2} &= \frac{C_0 (1-s)}{3} - 8 C_0 \beta_4\\
 	\Leftrightarrow ~~ C_0 \beta_4 &= \frac{C_0}{48} - \frac{1}{64}.
\end{align*}
Plugging the expressions for $\beta_2$ and $\beta_4$ in \eqref{eq:beta_square1} we get
\begin{align*}
	\frac{1}{16} - C_0 \left( \frac{1-2s}{6} \right) + C_0^{2}\left( \frac{1-2s}{3} \right)^{3} &= \frac{C_0 s}{3} + \frac{C_0}{12} - \frac{1}{16}.
\end{align*}
Rearranging terms we obtain
\begin{equation}
\label{eq:quad_C0}
	\frac{1}{8} - \frac{C_0}{4} + C_0^{2}\left( \frac{1-2s}{3} \right)^{2} = 0,
\end{equation}
whence
\[
	C_0 = \left( \frac{3}{1-2s} \right)^{2} \left( \frac{1}{8} - \frac{1}{2}\sqrt{\frac{1}{16} - \frac{1}{2} \left( \frac{1-2s}{3} \right)^{2}} \right),
\]
which is precisely \eqref{eq:C0}.

Now, we will write $\Phi$ in a more compact form. For any $x\in \{\Phi>0\}$ we have
\begin{align*}
	\Phi(x) = C_0 - \left( \left( \frac{1}{4} - \beta_2 \right)x_1^{2} + \left( \frac{1}{4} + \beta_2 \right) x_2^{2} \right) + \left( \frac{s}{12} + \beta_4 \right)x_1^{4}  + \left( \frac{1-s}{12} + \beta_4 \right)x_2^{4} - 6\beta_4 x_1^{2}x_2^{2}, 
\end{align*}
which by \eqref{eq:beta_system} may be written as
\begin{align*}
	\Phi(x) = C_0 - \left( \left( \frac{1}{4} - \beta_2 \right)x_1^{2} + \left( \frac{1}{4} + \beta_2 \right) x_2^{2} \right) + \frac{1}{4C_0}\left( \left( \frac{1}{4} - \beta_2 \right)x_1^{2} + \left( \frac{1}{4} + \beta_2 \right) x_2^{2} \right)^{2}.
\end{align*}

Since $t\mapsto C_0 - t + \frac{1}{4C_0}t^{2}$ is a nonnegative quadratic polynomial with zero at $t_0 = 2C_0$, we conclude that
\[
	\{\Phi>0\} = \left\{ \left( \frac{1}{4} - \beta_2 \right)x_1^{2} + \left( \frac{1}{4} + \beta_2 \right) x_2^{2} < 2C_0\right\}.
\]

	We would like to compute a formula for the limit $ \frac{\beta_M}{M^{1/3}} $ when we assume $ \{g = \gmax \} $ is a finite set of nondegenerate maxima. Note that this limit exists by applying Theorem \ref{mainthm:isolated_hom} at each maximum and by imposing that $ M $ equals the sum of the masses around each point, as we do in the sequel. This illustrates a way in which information about a \textit{global} parameter may be derived from local information at each point.

	Take $ M_n \rightarrow 0 $. Fix disjoint neighborhoods $ \{\Theta_i\}_{i=1}^N $ around each $ p_i $. For sufficiently large $ n $, we have $ \{u_n > 0 \} = \bigcup_{i=1}^N \Theta_i $ and
\begin{equation}
	\label{eq:masses}
	\sum_{i=1}^{N} \int_{\Theta_i} u_n \mathrm{d}S = M_n.
\end{equation}

Choose charts $x_i: \Theta_i \rightarrow B_{R_i}$ such that $x_i(p_i) = 0$. Around each $p_i$, it holds 
\[
g(x_i) = \gmax - x_i^{\top}A_i x_i + r_i(x_i),
\]
where $|r_i(x_i)| \leq C|x_i|^{3}$. Moreover, we can assume, up to making each $\Theta_i$ smaller, that
\begin{align*}
	g(x_i) \leq \gmax - \frac{1}{2}x_i^{\top}A_i x_i.
\end{align*}

By repeating the arguments in Section \ref{sect:hom_subseq} at each $ p_i $, we may choose subsequences such that 
\[
	\frac{1}{\beta_{M_n}^{3}}\int_{\Theta_i} u_{n} \rightarrow \overline{\alpha}_i ,~~ i=1,\hdots N.
\]
Therefore, \eqref{eq:masses} implies that $M_n^{-1/3} \beta_{M_n}$ converges to some $ \overline{\alpha} >0 $.

Define the rescalings 
\[
	v^{n}_i(y)= M_n^{-2/3} u_{n}(x_i^{-1}(M_n^{1/6}y)),
\]
which satisfy equations of the form \eqref{eq:U}. Hence, there exist $ \{r_i\}_{i=1}^{N}$ such that $ \supp(v^{n}_i) \subset B_{r_i} $ for all $ n $.

In terms of $v_i^{n}$, equation \eqref{eq:masses} becomes
\begin{equation}
	\label{eq:masses_vn}
	\sum_{i=1}^{N} \int_{B_{r_i}}v^{n}_i \mu_i^{n}\mathrm{d}y  = 1,
\end{equation}
where $ \mu^n_i $ is the surface density in a coordinate system defined in $ \Theta_i $.

Theorem \ref{mainthm:isolated_hom} implies that $v_{i}^{n} \rightharpoonup v_i$ weakly in $ W^{2,2}(\RR^{2}) $ and strongly in $ L^{1}(\RR^{2}) $. Hence, we can take the limit in \eqref{eq:masses_vn} and obtain 
\begin{equation}
\label{eq:mass_limit}
\sum_{i=1}^{N} \int_{B_{r_i}}v_i \mathrm{d}y = 1
\end{equation}

The equation satisfied by $v_i$ is 
\[
	-\Delta v_i = \left( \overline{\alpha} \gmax - \frac{1}{\gmax} y^{\top} A_i y \right) \mathbb{1}_{\{v_i>0\}}.
\]
Define $s_i \in (0, 1/2]$ as 
\[ 
	\frac{1}{\text{tr}(A_i)} A_i \sim \begin{pmatrix}
		s_i & 0 \\
		0 & 1 - s_i
\end{pmatrix}
\]
up to an orthogonal transformation. Then, 
\[
	v_i(y) = \frac{\overline{\alpha}^{2} \gmax^{3}}{\text{tr}(A_i)} \Phi \left( \sqrt{\frac{\text{tr}(A_i)}{\gmax^{2} \overline{\alpha}}} y ;~ s_i\right) 
\]
where $\Phi(z;~s_i)$ is the solution to \eqref{eq:obst_reduced} with $\Tilde{A} = \begin{pmatrix} s_i & 0 \\ 0 & 1-s_i \end{pmatrix}$. Moreover, we define
\begin{equation}
	\label{eq:mass_def}
	M(s) := \int_{\RR^{2}} \Phi(x;s) \mathrm{d}x. 
\end{equation}

Plugging the expression for $v_i$ in \eqref{eq:mass_limit} yields
\begin{align*}
	\sum_{i=1}^{N} \frac{\gmax^{5} \overline{\alpha}^{3}}{\text{tr}(A_i)^{3/2}} \int_{\RR^{2}} \Phi(z;~ s_i) \mathrm{d}z = 1, 
\end{align*}
from which we recover an expression for $\overline{\alpha}$
\[
	\overline{\alpha} = \left( \sum_{i=1}^{N} \frac{\gmax^{5} M(s_i)}{\text{tr}(A_i)^{3/2}} \right)^{-1/3}.
\]
Putting this formulas together yields \eqref{eq:lambda_i} 

To conclude this subsection, we would like to express the mass of $\Phi$ as a function of $s$. This will let us compare the concentration of $u$ around maximum points of the same homogeneity, using formulas \eqref{eq:v_hom} and \eqref{eq:mass_limit}.

\begin{lem}
	The mass at a point $M(s)$ defined by \eqref{eq:mass_def} satisfies equation \eqref{eq:mass}. Moreover, the function $s\mapsto M(s)$ is decreasing.
\end{lem}
\begin{proof}
Denote $k_1 = \frac{1}{4} - \beta_2$ and $k_2 = \frac{1}{4} + \beta_2$. A direct computation yields
\begin{align*}
	\int_{\RR^{2}} \Phi \mathrm{d}x &= \int_{\{\Phi>0\}} C_0 - \left( k_1x_1^{2} + k_2 x_2^{2} \right) + \frac{1}{4C_0}\left( k_1x_1^{2} + k_2 x_2^{2} \right)^{2} \mathrm{d}x \\
				     &= C_0 \int_{\{\Phi>0\}} 1 - \left( \left( \sqrt{\frac{k_1}{C_0}} x_1 \right)^{2} + \left( \sqrt{\frac{k_2}{C_0}} x_2 \right)^{2} \right) + \frac{1}{4}\left( \left( \sqrt{\frac{k_1}{C_0}} x_1 \right)^{2} + \left( \sqrt{\frac{k_2}{C_0}} x_2 \right)^{2} \right)^{2} \mathrm{d}x.
\end{align*}
Define $y_1 = \sqrt{\frac{k_1}{C_0}} x_1$,	$ y_2 = \sqrt{\frac{k_2}{C_0}} x_2$, then, $\{\Phi > 0\} = \left\{ y_1^{2}+ y_2^{2} < \frac{1}{2} \right\} $. Hence, by the change of variables formula
\[
	M(s) = \frac{C_0^{2}}{\sqrt{k_1 k_2}} \int_{B_{\sqrt{2}/2}} 1 - |y|^{2} + \frac{|y|^{2}}{4} \mathrm{d}y = \frac{\kappa C_0^{2}}{\sqrt{k_1 k_2}},  
\]
where $\kappa := \int_{B_{\sqrt{2}/2}} 1 - |y|^{2} + \frac{|y|^{2}}{4} \mathrm{d}y = \frac{73\pi}{96}$, which is exactly \eqref{eq:mass}. Differentiating \eqref{eq:C0} with respect to $s$, yields, after some algebraic manipulation
\[
	C_0'(s) = \frac{9}{2(1-2s)^{3}} \left( \sqrt{\frac{1}{16} - \frac{1}{18} (1-2s)^{2}} - \left( \frac{1}{4} + \frac{7}{9}(1-2s)^{2} \right) \right) \left( \frac{1}{16} - \frac{1}{18}	(1-2s)^{2} \right)^{-1/2} 
\]
which is nonpositive for all $s \in [0, \frac{1}{2}]$. On the other hand, we have 
\begin{align*}
	\beta_2'(s) &= \frac{1}{3} \left( (1-2s) C'_0 -2C_0 \right) \leq 0, \\
	k_1'(s) &= -k_2'(s) = -\beta_2'(s). 
\end{align*}

Differentiating $ M(s) $ yields
\[
M'(s) = \kappa \left( \frac{2C_0 C_0'(s)}{\sqrt{k_1 k_2}} - \frac{C_0^{2}}{2 \left( k_1k_2 \right)^{3/2}} \left( k_1 \beta_2' - k_2 \beta_2' \right)  \right) =  \kappa \left( \frac{2C_0 C_0'(s)}{\sqrt{k_1 k_2}} + \frac{C_0^{2} (k_2 - k_1)}{2 \sqrt{k_1k_2}} \beta_2'(s)  \right) \leq 0
\]
and thus the mass is decreasing with $s$.
	
\end{proof}

\subsection{Regularity of the free boundary}

Let us now prove the regularity of the free boundary for small mass, as claimed in Theorem \ref{mainthm:morse}.

First, let us note that, near the free boundary of the limit solution, the right hand side of \eqref{eq:obst_reduced} is nondegenerate, that is, it is bounded above by a negative constant. For this, it suffices to compare the semiaxes of the ellipses $ \{ sx_1^2 + (1-s)x_2^2  = 1 \} $ and $ \left\{ \left( \frac14 - \beta_2\right) x_1^2 + \left( \frac14 + \beta_2 \right)x^2_2 = 2C_0 \right\}. $

Inspecting equation \eqref{eq:C0}, it is easy to see that $ C_0 $ may be extended continuously to a function defined for all $ s\in (0,1) $, with the property that $ C_0(s) = C_0(1-s) $ for every $ s $. A direct computation yields
\begin{align*}
	\frac{d}{ds}\left( \frac{1}{2C_0} \left(\frac14 - \beta_2\right) \right) &= - \left( \frac{2C_0 - C_0' (1-2s)}{6C_0} +\left(\frac14 - \beta_2\right) \frac{C_0'}{2C_0^2} \right) \\
										 &=\left( \frac{1}{3} - \frac{C'_0}{8C_0^{2}} \right). 
\end{align*}
Differentiating \eqref{eq:quad_C0} with respect to $ s $, we get the following formula for $ C_0' $:
\[C_0' = \frac{16C_0^2(1-2s)}{8C_0(1-2s)^2 - 9}.\]
Plugging this formula in the derivative above we get
\begin{align*}
	\frac{d}{ds}\left( \frac{1}{2C_0} \left(\frac14 - \beta_2\right) \right) &= \frac{1}{3} - \frac{2 (1-2s)}{8C_0(1-2s)^2 - 9}\\
										 &= \frac{1}{3} + \frac{2(1-2s)}{3 \sqrt{9 - (1 - 2s)^2}},
\end{align*}
where we used formula \eqref{eq:C0}. We see that $ \frac{d}{ds}\left( \frac{1}{2C_0} \left(\frac14 - \beta_2\right) \right)$ is decreasing in $ (0,1) $, therefore
\[ -1 < \frac{1}{3} \left(1 - \sqrt{2}  \right) \leq\frac{d}{ds}\left( \frac{1}{2C_0} \left(\frac14 - \beta_2\right) \right) \leq \frac{1}{3} \left(1 + \sqrt{2}  \right) < 1 , \]
This means that $ \frac{1}{2C_0}\left( \frac14 -\beta_2\right) - s $ is decreasing in $ (0,1) $. Moreover
\[
\lim_{s\rightarrow 0^+}\left( \frac{1}{2C_0s} \left(\frac14 - \beta_2\right) \right) = \lim_{s\rightarrow 0^+} \frac{(1-2s) (14 - 16s - 2 \sqrt{1 + 16s - 16s^2})}{3(3 -\sqrt{1 + 16s - 16s^2}) (1 + \sqrt{1 + 16s - 16s^2})} = 1,
\]
which means that $ \frac{1}{2C_0} \left(\frac14 - \beta_2\right) < s$ for $ s \in (0,\frac12) $. 

On the other hand, we have that 
\[\frac{1}{2C_0(s) (1-s)} \left( \frac{1}{4} + \beta_2(s) \right) = \frac{1}{2C_0(1-s) (1-s)} \left( \frac{1}{4} - \beta_2(1-s)\right)  \]
by the symmetry of $ C_0 $ and \eqref{eq:beta2}. Thus, we get that $ \frac{1}{2C_0} \left( \frac{1}{4} + \beta_2 \right) < 1-s $ for every $ s\in (0,\frac{1}{2}) $. 

Therefore, both semiaxes of $ \left\{ \left( \frac14 - \beta_2\right) y_1^2 + \left( \frac14 + \beta_2 \right)y_2 = 2C_0 \right\} $ are larger than those of $ \{ sy_1^2 + (1-s)y_2^2  = 1 \} $. Hence, there exists some $ \lambda_0 > 0$ such that $ y^{\top}\Tilde{A} y> 1 + \lambda_0 $ in a neighborhood of the free boundary. By rescaling and rotating, this holds for every $ A $ and $ \overline{\alpha} $.

In order to obtain the regularity claim, it suffices to apply Proposition \ref{prop:hom_sandwich} with $ \delta_0 > 0 $ such that $ \partial \{v_{\overline{\alpha} - \delta} \} $ is at a positive distance, say $ d $ to $ \left\{ \frac{\beta_M}{m^{1/3}} - y^{\top}A y + o(1) = 0 \right\} $ for $ m $ smaller that some $ m_0 $, which is possible given the above discussion.

Denote $S_0 =  \{v_{\overline{\alpha} + \delta_0}> 0 \}\setminus \{v_{\overline{\alpha} - \delta_0}> 0 \}$. Note that for any $ \delta \leq \delta_0 $ we have
\[
	\partial \{U_{m_{\delta}} > 0\} \subset \{v_{\overline{\alpha} + \delta}> 0 \}\setminus \{v_{\overline{\alpha} - \delta}> 0 \} =: S_{\delta} \subset S_0,
\]
where $ m_{\delta} $ is the one given by Proposition \ref{prop:hom_sandwich}. Fix any $ \varepsilon \in (0, \frac{1}{8}) $, and apply a rescaled version of Proposition \ref{prop:reg} in each ball of radius $ d/2 $ with center in $ S_0 $ to get some $ \rho_0 $ such that the conclusion of the proposition holds. Then, there is some $ \delta_1 $ such that, for any $y \in S_{\delta_1} $, we have
\[\frac{|\{U_{m} = 0\} \cap B_{\rho_0 / 2}(y)|}{|B_{\rho_0 / 2}(y)|} \geq \frac{|\{v_{\overline{\alpha} + \delta_1} = 0 \} \cap B_{\rho_0/2}(y) |}{|B_{\rho_0/2}(y)|} > \varepsilon\]
for any $ m\leq m_{\delta_1} $. This yields the $ C^1 $ regularity of $ \partial \{U_{m} > 0\} $ for small $ m $. 

\section{An isolated maximum described to the leading order by an nonvanishing anisotropic function}
\label{sect:inhom}

In this section, we prove Theorem \ref{mainthm:inhom}. The structure is similar to the preceding section, with additional attention given to the issue of convergence.

Choose a local chart $ (x,\Theta) $ such that $x(p) = 0$ and $ x(\Theta) = B_R $, and we may write 
\[g(x) = \gmax  - a x_1^{4} - b x_2^{2}+ r(x) \faskip \forall x\in B_R,\]
with $r(x) = o(x_1^{4} +x_2^{2}) $. Up to making the chart smaller, we may assume 
\begin{equation} 
	\label{eq:hypothesis_degenerate}
	g(x) \leq \gmax -\frac{a x_1^4 + b x_2^2}{2} \faskip \forall x\in B_R.
\end{equation}

Fix a sequence $ M_n \rightarrow 0 $. As in Section \ref{sect:hom}, we retain the notation for $ u_n $, $ \alpha_n $ and so on.

For $n$ large enough we have $u_n |_{\partial B_r} = 0$, and for $ x= (x_1, x_2) \in B_r $, denote \[y = (m_n^{-1/11} x_1, m_n^{-2/11}x_2) \] and define the rescaling
\[
	u_n(x_1,x_2) = \gmax^2 m_n^{8/11} U_n(\gamma^{-1/4}m_n^{-1/11} x_1,  \gamma^{-1/2}m_n^{-2/11} x_2).
\]
In this case, $U_n$ is supported in the set
\[
	O_n:= \{y:~ (\gmax^{1/4}m_n^{1/11} y_1)^{2} + (\gmax^{1/2}m_n^{2/11}y_2)^{2} < R\}.
\]
For the rest of this section, we will denote $\varepsilon_n := m_n^{1/11}$.

As in Section \ref{sect:hom}, we obtain that 
\begin{align*}
	-\divg( H_{n}(y) \nabla U_{n}(y)) = \left( \frac{\beta_{n}}{\varepsilon_n^{4}} - ay_1^{4} - by_2^{2} + R_n(y)\right)\mathbb{1}_{\{U_{n} >0 \}} ~ \text{ in } O_{n} 
\end{align*}
where 
\begin{align*}
	H_{n}(y) &= \begin{pmatrix}
			\gmax^{1/2}\varepsilon_n^{2} h_{11}(\varepsilon_n y_1, \varepsilon_n^{2} y_2) & \gmax^{1/4}\varepsilon_n h_{12}(\varepsilon_n y_1, \varepsilon_n^{2} y_2) \\
			\gmax^{1/4}\varepsilon_n h_{12}(\varepsilon_n y_1, \varepsilon_n^{2} y_2) &  h_{22}(\varepsilon_n y_1, \varepsilon_n^{2} y_2)
		\end{pmatrix}, \\ 
		R_n(y) &= \frac{R(\gmax^{1/4}\varepsilon_n y_1, \gmax^{1/2}\varepsilon_n^{2} y_2)}{\varepsilon_n^{4}}.
\end{align*}
Thus, for any $K \subset \RR^{2}$ compact and $ \varphi \in C^{2}(K) $, we have 
\begin{equation}
	\label{eq:deg_asympt}
	-\divg(H_n \nabla \varphi) = -h_{22}(0) \partial_{22}\varphi + O(\varepsilon_n) ~ \text{ in } K.
\end{equation}

For the subsequent analysis, we remark that, without loss of generality, we may take $ a = b = 1 $. Indeed, it suffices to modify the rescaling as
\[u_n(x_1, x_2) = \gmax^2 (m_n^{8} a^{2}b)^{1/11} \Tilde{U}_n\left( \left( \frac{b^{3/4}}{a^{5/4}}m_n \right)^{1/11} x_1, \left( \frac{a^{1/4}}{b^{5/4}}m_n \right)^{2/11} x_2  \right), \]
and a direct computation shows that 
\[-\divg(\Tilde{H}_n \nabla \Tilde{U}_n) = \left( \frac{\beta_n}{(ab^{6})^{1/11} m_n^{4/11}} - y_1^{4} - y_2^{2} + \Tilde{R}_n(y) \right)\mathbb{1}_{\{\Tilde{U}_n > 0\}}, \]
with $ \Tilde{H}_n $, $ \Tilde{R}_n $ defined in an analogous way, and $ \Tilde{H}_n $ also satisfying \eqref{eq:deg_asympt}. In a similar way, we can assume $ h_{22}(0) = 1 $. 

\subsection{The limit problem}
\label{ssect:limit_deg}

We will first study the limit problem, which in view of \eqref{eq:deg_asympt} should have the form 
\begin{equation}
	\label{eq:deg_limit_prob}
	\begin{cases}
		- \partial_{y_2 y_2 } V = \left( \overline{\alpha} - y_2^{2} - y_1^{4} \right) \mathbb{1}_{\{V>0\}} & \text{ in } \RR^{2}, \\ 
		V \geq 0, ~~ V \in L^{1}(\RR^{2}),
	\end{cases}
\end{equation}
for some fixed $ \overline{\alpha} > 0 $. This problem can be brought into the normalized form 
\begin{equation}
\label{eq:limit_deg}
\begin{cases}
	- \partial_{y_2 y_2 } \Phi = \left( 1 - y_2^{2} - y_1^{4} \right) \mathbb{1}_{\{\Phi>0\}} & \text{ in } \RR^{2}, \\ 
	\Phi \geq 0, ~~ \Phi \in L^{1}(\RR^{2}),
\end{cases}
\end{equation}
by defining
\begin{equation}
	\label{eq:deg_V}
	V(y) = \overline{\alpha}^{2} \Phi \left( \frac{1}{\overline{\alpha}^{1/4}} y_1, \frac{1}{\overline{\alpha}^{1/2}} y_2 \right).
\end{equation}

Let us first state the following uniqueness result:
\begin{lem}
	\label{lem:deg_uniqueness}
	There is at most one solution $ V \in L^{2}(\RR^{2}) $ to \eqref{eq:deg_limit_prob} such that $ \partial_2 V\in L^{2}(\RR^{2}) $.
\end{lem}
\begin{proof}
	The weak formulation of \eqref{eq:limit_deg} reads
	\begin{equation}
		\label{eq:deg_limit_weak}
		\int_{\RR^{2}} \partial_2 V \partial_2(\varphi - V) \geq \int_{\RR^{2}} \left( 1 - y_2^{2} - y_1^{4} \right) (\varphi - V) \faskip \forall \varphi \in L^2(\RR^{2}), \text{ s. t. } \partial_2 \varphi \in L^2(\RR^2),~ \varphi \geq 0.
	\end{equation}
	If $ V_1, V_2 $	are two solutions, then, by testing their corresponding equations with each other, we get 
	\[\int_{\RR^{2}} (\partial_2(V_1 - V_2))^{2} \mathrm{d}y = 0, \]
in other words, $ \partial_2 (V_1 - V_2) = 0 $ a. e. in $ \RR^2 $.

Hence, there exists $ f:\RR \rightarrow \RR $ such that $V_1(y_1, y_2) = V_2(y_1,y_2) + f(y_1) $. Thus, $f = V_1 - V_2 \in L^{2}(\RR^{2})$. Suppose $f\not \equiv 0$, then
\[\int_{\RR} \int_{\{f > 0\}} (f(y_1))^2 \mathrm{d}y_1 \mathrm{d}y_2 = \infty,\]
a contradiction. Therefore $ f\equiv 0 $ and $ V_1 = V_2 $ a.e. in $\RR^2$.
\end{proof}

Note that \eqref{eq:limit_deg} can be solved explicitly. Indeed, in $\{\Phi > 0\}$ one must have 
\begin{equation}
	\label{eq:Phi_deg1}
	\Phi(y_1, y_2) = \frac{y_2^{4}}{12} +(y_1^{4} -1) \frac{y_2^{2}}{2} + f(y_1),
\end{equation}
for some $f: \RR \rightarrow \RR$ to be determined. On $\partial \{ \Phi>0\}$ we have
\begin{align*}
	\frac{y_2^{4}}{12} + ( y_1^{4} -1) \frac{y_2^{2}}{2} + f(y_1) &= 0, \\
	\frac{y_2^{3}}{3} + ( y_1^{4} -1) y_2 &= 0,
\end{align*}
which suggests $\partial \{ \Phi>0\} = \left\{\frac{ y_2^{2}}{3} + y_1^{4} = 1 \right\}$. On this set, the first equation can be rewritten as
\[
	f(y_1) = \frac{ y_2^{4}}{4},
\]
but, written in terms of $y_1$, this becomes
\[
	f(y_1) = \frac{3}{4}(1 -  y_1^{4})^{2}.
\]

Hence, we can plug this expression in \eqref{eq:Phi_deg1} to get
\[
	\Phi(y_1,y_2) = \frac{3}{4} - \frac{3}{2} \left( \frac{y_2^{2}}{3} +  y_1^{4} \right) + \frac{3}{4} \left( \frac{y_2^{2}}{3} + y_1^{4} \right)^{2}.
\]
One can readily see that $\Phi > 0$ in $\left\{\frac{ y_2^{2}}{3} +  y_1^{4} < 1\right\}$. Thus, the solution to \eqref{eq:limit_deg} is given by 
\begin{equation}
	\label{eq:deg_sol}
	\Phi(y_1, y_2) = \begin{cases}
	\frac{3}{4} - \frac{3}{2} \left( \frac{ y_2^{2}}{3} + y_1^{4} \right) + \frac{3}{4} \left( \frac{y_2^{2}}{3} + y_1^{4} \right)^{2} & \text{ in } \left\{\frac{y_2^{2}}{3} + y_1^{4} \leq 1\right\}, \\
	0 & \text{ outside}.
\end{cases}
\end{equation}

Define $ V_{\overline{\alpha}} $ by the left-hand side of \eqref{eq:deg_V} where $ \Phi $ is defined by \eqref{eq:deg_sol}. Note that, for $ \alpha, \beta \in [c_0^{-1}, c_0] $, there exists $C = C(c_0) > 0 $ such that 
\begin{equation}
\label{eq:cont_infty}
	\|V_{\alpha} - V_{\beta}\|_{L^{\infty}(\RR^{2})} \leq C |\alpha - \beta|.
\end{equation}

\subsection{Convergence}

In this subsection we will justify that $ (U_n)_{n\in\NN} $ converges in the sense specified in Theorem \ref{mainthm:inhom} to a solution to \eqref{eq:deg_limit_prob}. As in Section \ref{sect:hom_subseq}, we will show uniform estimates on the leading term in the right hand side and on $ (U_n)_{n\in \NN} $. 

\begin{prop}
	\label{prop:deg_bound_q}
	There exists $C_0> 0$ such that $ \frac{\beta_{n}}{\varepsilon_n^{4}} \leq C_0 $ for all $n \in \NN$.
\end{prop}
\begin{proof}
	By contradiction, take a subsequence such that $\frac{\beta_{n}}{\varepsilon_n^{4}} \rightarrow \infty$.

	We claim there exist $c_0, ~r_0 > 0$ such that for any $n$ sufficiently large	
	\[
		U_{n} \geq c_0 \frac{\beta_{n}}{\varepsilon_n^{4}}  \faskip \text{in } (-r_0, r_0)^{2}.
	\]

	Indeed, fix $r_0 > 0$ such that 
	\[
		\frac{\beta_{n}}{\varepsilon_n^{4}} - a y_1^{4} - b y_2^{2} + R_{n}(x) \geq \frac{\beta_{n}}{2\varepsilon_n^{4}} \faskip \text{in } (-2r_0, 2r_0)^{2},
	\]
	and define
	\[w(y) = \frac{\beta_{n}}{8\varepsilon_n^{4}}(4 r_0^{2} - y_2^{2}) + \psi_{n}(y),\]
	where $\psi_{n}$ solves
	\[
		\begin{cases}
			-(\varepsilon_n^{2} \partial_{11} + \partial_{22}) \psi = 0 & \text{in }(-2r_0, 2r_0)^2, \\
			\psi(y_1, \pm 2r_0) = 0 & \forall y_1 \in (-2r_0, 2r_0), \\
			\psi(\pm 2 r_0, y_2) = \frac{\beta_{n}}{4\varepsilon_n^{4}}(y_2^{2} - 4 r_0^{2}) & \forall y_2 \in (-2r_0, 2r_0).  
		\end{cases}
	\]
	Note that the boundary datum is continuous and the boundary of the domain is Lipschitz regular, hence the Dirichlet problem has a classical solution.

	Using separation of variables, we obtain the representation formula:
\begin{align*}
	\psi_{n}(y) = \frac{16r_0^{2}\beta_{n}}{\pi^{3}\varepsilon_n^{4}}\sum_{k=1}^{\infty} \frac{(-1)^{k}}{(2k+1)^{3} } \frac{\cosh\left(\frac{(2k+1)\pi}{4r_0\varepsilon_n}y_1\right)}{\cosh\left(\frac{(2k+1)\pi}{2\varepsilon_n}\right)}	\cos \left( \frac{(2k+1)\pi}{4r_0} y_2 \right),
\end{align*}
from which we obtain 
\begin{equation}
	\label{eq:bound_psi}
	|\psi_{n}(y)| \leq C \frac{\beta_{n}}{\varepsilon_n^{4}}e^{-\pi/4\varepsilon_n} \faskip \forall (y_1, y_2) \in (-r_0, r_0)^{2}.
\end{equation}

By \eqref{eq:deg_asympt}, we have 
\[-\divg(H_n\nabla w) = \frac{\beta_n}{4\varepsilon_n^{4}} + o(1)  ~\text{ in } (-2r_0,2r_0)^{2},\]
so that, for $n$ large, $-\divg(H_n \nabla w) \leq \frac{\beta_n}{2 \varepsilon_n^{4}}$ in $ (-2r_0, 2r_0)^{2} $.

Now, $U_{n}$ is a nonnegative supersolution to $- \divg(H_n\nabla U_{n}) = \frac{\beta_{n}}{2\varepsilon_{n}^{4}}$ in $(-2r_0, 2r_0)$, so $U_{n} \geq w$ in $ (-2r_0, 2r_0) $ by the maximum principle. Thus, 
\[U_{n} \geq \frac{\beta_{n}}{\varepsilon_n^{4}}\left( \frac{3}{8} - Ce^{-\pi/4\varepsilon_n} \right) ~ \text{ in } (-r_0, r_0). \]
Choose $ n_0 $ such that the right hand side is larger than $\frac{\beta_{n}}{2\varepsilon_n^{4}}$, and the claim follows for $c_0 = \frac{1}{2}$.

Finally, we have
\[1 = \int_{\RR^{2}} U_{n} \geq \frac{(2r_0)^{2} c_0 \beta_{n}}{\varepsilon_n^{4}} \rightarrow +\infty, \]
a contradiction. The claim of the proposition follows.
\end{proof}

\begin{prop}
	\label{lem:deg_infty_bd}
	There exist $C_1 , R_1 > 0$ such that $ \supp(U_n) \subset B_{R_1}$ and $ \|U_n\|_{L^{\infty}} \leq C_1 \left( \frac{\beta_n}{\varepsilon_n^{4}} \right)^2 $.
\end{prop}

\begin{proof}
	Let $ \varphi $ be defined by
	\[
		\varphi(y) = 4 \left( \frac{\beta_n}{\varepsilon_n^{4}} \right)^2 \Phi \left( \left( 4 \frac{\beta_n}{\varepsilon_n^{4}}\right)^{-1/4} y_1,  \left( 4 \frac{\beta_n}{\varepsilon_n^{4}}\right)^{-1/2} y_2 \right) 
	\]
	which has support in $ F_n := \left\{ y_1^{4} + \frac{y_2^{2}}{3} < \frac{4 \beta_n}{\varepsilon_n^{4}} \right\} $. A straightforward computation shows that $ \varphi $ satisfies 
	\begin{align*}
		-\partial_{22}\varphi &= \left( \frac{\beta_n}{\varepsilon_n^{4}} - \frac{1}{4}(y_1^{4} + y_2^{2}) \right) \mathbb{1}_{\{\varphi > 0 \}}.
	\end{align*}
	Therefore, 
	\begin{align*}
		-\divg(H_n\nabla \varphi) &\geq \frac{\beta_n}{\varepsilon_n^{4}} - \frac{1}{4}(y_1^{4} + y_2^2) + o(1) \\
					  &\geq \frac{\beta_n}{\varepsilon_n^{4}} - \frac{1}{2} (y_1^{4} + y_2^2)
	\end{align*}
	in $ (-\varepsilon_n^{-1}R, \varepsilon_n^{-1}R) \times (-\varepsilon_n^{-2}R, \varepsilon_n^{-2}R) $, which contains $ O_{n} $. 

	Since 
	\begin{align*}
		U_n(y_1, \pm \varepsilon_n^{-2}R) = 0 = \varphi & \faskip \forall y_1 \in (-\varepsilon_n^{-1}R, \varepsilon_n^{-1}R),\\ 
		U_n(\pm \varepsilon_n^{-1}R, y_2) = 0 = \varphi & \faskip \forall y_2 \in(-\varepsilon_n^{-2}R, \varepsilon_n^{-2}R),
	\end{align*}
	we can apply Lemma \ref{lem:mp} to get $ \varphi \geq U_n $.  

	By Proposition \ref{prop:deg_bound_q}, there is $ {R_1} $ such that $ F_n \subset B_{R_1} $ for all $ n $. In particular, for $ y \in \RR^{2} \setminus B_{R_1} $, $U_n \equiv 0$, where $ C_0 $ is that of the preceding Proposition. Moreover, $ U_n \leq \|\varphi\|_{L^{\infty}} = 3 \left( \frac{\beta_n}{\varepsilon_n^{4}} \right) ^{2}$, and the Lemma follows.
\end{proof}

\begin{cor}
	\label{cor:deg_bounds}
	$ \liminf_{n\rightarrow \infty} \frac{\beta_{n}}{\varepsilon_n^{4}} > 0 $
\end{cor}

\begin{proof}
	By Proposition \ref{lem:deg_infty_bd} 
	\[1 = \int_{B_{R_1}}U_n \mu_n(y) \mathrm{d}y \leq \mu_n(B_{R_1})\left( \dfrac{\beta_{\varepsilon_n}}{\varepsilon_n^{4}} \right)^{2},\]
which implies the claim by recalling that $\mu_n(B_{R_1}) \rightarrow |B_{R_1}| > 0$.
\end{proof}

The next Proposition will give simultaneously the a.e. convergence of $ U_n $ and the convergence of the interfaces. Its proof is identical to that of Proposition \ref{prop:hom_sandwich}.
\begin{prop}
	\label{prop:sandwich}
	Assume $ \frac{\beta_n}{\varepsilon_n^{4}} \rightarrow \overline{\alpha} $. For every $\delta \in (0, \overline{\alpha})$, there exists $ n_0 \in \NN $ such that 
	\begin{equation}
	\label{eq:deg_sandwich}
		V_{\overline{\alpha} - \delta} \leq U_n \leq V_{\overline{\alpha} + \delta} \text{ in } \RR^{2}.
	\end{equation}
	for every $n\geq n_0$.
\end{prop}

\begin{proof}[Proof of Theorem \ref{mainthm:inhom}]
	By Proposition \ref{prop:deg_bound_q} and Corollary \ref{cor:deg_bounds}, we may take a subsequence such that $\frac{\beta_n}{\varepsilon_n^{\gamma}} \rightarrow \overline{\alpha} > 0$.

	For this subsequence, Proposition \ref{prop:sandwich} and \eqref{eq:cont_infty} implies that $ U_n \rightarrow V_{\overline{\alpha}} $ almost everywhere in $ \RR^{2} $. Since $ \supp(U_n) \subset B_{R_1} $ and $ U_n $ is bounded in $ L^{\infty}(\RR^{2}) $, $U_n \rightarrow V_{\overline{\alpha}}$ strongly in $ L^{p}(B_{R_1}) $ for every $ p\in [1,\infty) $.
	
	On the other hand, recall that $ U_n $ minimizes
	\[\min \left\{ J_n(v) = \int_{B_{R_1}} \frac{|H_n \nabla v|^{2}}{2}  - \left( \frac{\beta_n}{\varepsilon_n^{4}} - y_1^{4} - y_2^{2} \right) v :~ v\in H^{1}_0(B_{R_1}),~ v\geq 0  \right\}. \]
	Fix $ \varphi \in C^{2}_c(B_{R_1 / 2}) $ with $ \varphi \geq 0 $, then
	\[J_n(U_n) \leq J_n(\varphi) \leq C.\]
	Hence,
	\begin{align*}
		\int_{B_{R_1}} (\partial_2 U_{n})^2 \mathrm{d}y &\leq CJ_n(U_n) + \int_{B_{R_1}} \left( \frac{\beta_n}{\varepsilon_n^{4}} - y_1^{4} - y_2^{2} + R_n(y) \right) U_n \mathrm{d}y  \\ 
								&\leq C + \|U_n\|_{L^{\infty}(B_{R_1})} \int_{B_{R_1}}\left| \frac{\beta_n}{\varepsilon_n^{4}} - y_1^{4} - y_2^{2} + R_n(y) \right| \mathrm{d}y \\
								&\leq C.
	\end{align*}
	Thus, $ (\partial_2 U_n) $ is a bounded sequence in $ L^{2}(B_{R_1}) $, whence there exists $ \overline{V} \in L^{2}(B_{r_1}) $ and a subsequence such that $\partial_2 U_n \rightharpoonup \overline{V}$ weakly in $ L^{2}(B_{r_1}) $. Using standard arguments we see that in fact $ \overline{V} = \partial_2 V_{\overline{\alpha}} $.

	The convergence along the whole parameter follows by an analogous argument as for Theorem \ref{mainthm:isolated_hom}, since solutions of the limit problem are also unique.
\end{proof}

\section{An isolated maximum described to the leading order by a function with nontrivial zero set}
\label{sect:noncoercive}

In this section we will prove Theorem \ref{mainthm:noncoercive}. We have
\[
	g(x) = \gmax - x_1^{4} - x_1^{2} x_2^{2} - x_2^{6} + r(x),
\]
in Riemannian normal coordinates, i.e. coordinates such that \eqref{eq:metric} holds. Again, we may choose a smaller neighborhood in order to have
\[
	g(x) \leq \gmax - \frac{1}{2} \left( x_1^{4} + x_1^{2} x_2^{2} + x_2^{6} \right). 
\]
Without loss of generality, we may also assume $ \inf_{B_R} h_{11} > 0 $.

In this case, there are two plausible scalings, namely
\begin{subequations}
	\label{eq:scalings}
	\begin{align}
		u_M(x_1, x_2) &= m^{3/4} U_m(m^{-1/8}x_1, m^{-1/8} x_2), \label{eq:sc_hom} \\
		u_M(x_1, x_2) &= m^{10/13} V_m(m^{-2/13} x_1,m^{-1/13} x_2), \label{eq:sc_inhom} 
	\end{align}
\end{subequations}
which are chosen in order to have $ \int_{\RR^{2}} U_m  = \int_{\RR^{2}} V_m = 1 $. By the same computations as in Sections \ref{sect:hom} and  \ref{sect:inhom}, we see that the leading term in the right hand side of the equation for $ U_m $ is $ \frac{\beta_M}{m^{1/2}} $ and that of $ V_m $ is $ \frac{\beta_M}{m^{6/13}} $. Thus, since $ m < 1$, we have
\[\dfrac{\frac{\beta_M}{m^{6/13}}}{\frac{\beta_M}{m^{1/2}}} = m^{1/26} \rightarrow 0.\]
Hence, it will suffice to show that $ \frac{\beta_M}{m^{1/12}} $ is bounded to conclude that $ \frac{\beta_M}{m^{6/13}} \rightarrow 0 $, and that therefore $ V_m $ will converge to zero weakly in $ L^{p} $. 

Take $ M_n\rightarrow 0 $, and denote $ \varepsilon_n = m_n^{1/8} $, thus, with the same notation as in Section \ref{sect:hom}, satisfies
\begin{align*}
	\int_{B_{R_n}} U_n \mu_n(y) \mathrm{d}y &= 1, \\ 
	-\divg(H_n(y) \nabla U_n) &= \left( \frac{\beta_n}{\varepsilon_n^{4}} - y_1^{4} - y_1^{2}y_2^{2} - \varepsilon_n^{2} y_2^{6} +o(1)  \right) \mathbb{1}_{\{U_n >0 \}}.
\end{align*}

This suggests that the limit equation should be
\begin{equation}
\label{eq:obst_noncoercive}
\begin{cases}
	-\Delta v = \left( \overline{\alpha} - x_1^{4} - x_1^{2} x_2^{2} \right) \mathbb{1}_{\{v>0\}} & \text{ in }\RR^{2},\\
	v\geq 0, v\in L^{1}(\RR^{2}). &
\end{cases}
\end{equation}
The right-hand side lacks the strict positivity to replicate the argument in Proposition \ref{prop:existence}. In Figure \ref{fig:noncoercive} we give an illustration of the zero set of the right-hand side of the equation.

\begin{figure}[ht]
	\centering
	\includegraphics[width=0.6\textwidth]{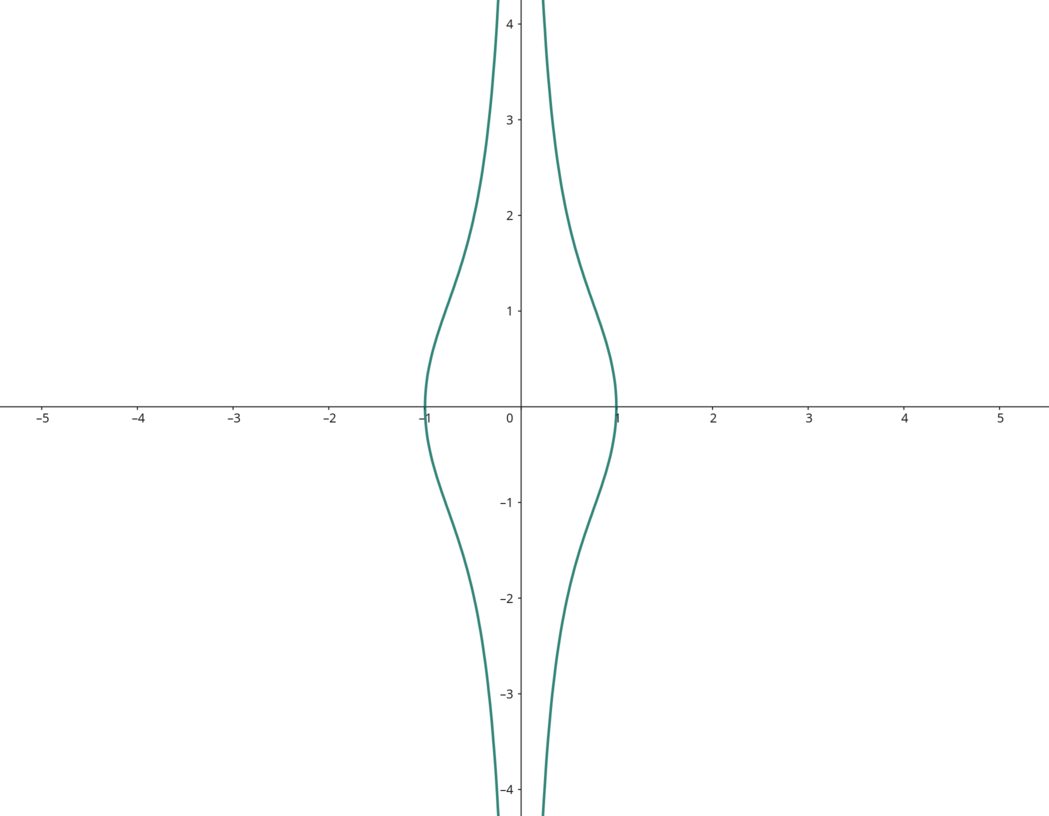}
	\caption{A plot of the set $ \{x_1^4 + x_1^2x_2^2 = 1\} $}
	\label{fig:noncoercive}
\end{figure}

Via an approximating argument, we can prove existence of a solution to the limit problem:
\begin{prop}
	\label{prop:exist_noncoercive}
	There exists a unique solution $ v \in H^{1}(\RR^{2}) $ to \eqref{eq:obst_noncoercive}.
\end{prop}
\begin{proof}
	The uniqueness part follows by the same argument as in the proof of Proposition \ref{prop:existence}.

	For $ \delta \in (0,1) $, define the problem
	\begin{equation}
	\label{eq:min_delta}
		\min \left\{ J_{\delta}(v) := \int_{\RR^{2}} \frac{|\nabla v|^{2}}{2} + \left( x_1^{4} + x_1^{2}x_2^{2} + \delta x_2^{4} \right)  v \mathrm{d}y: v\in H^{1}(\RR^{2}),~ v\geq 0, ~ \int_{\RR^{2}}v = 1  \right\}.
	\end{equation}

	We conclude that $ v_{\delta} $ is the unique minimizer of \eqref{eq:min_delta}. Hence, it satisfies
	\begin{equation}
		\label{eq:obstacle_delta}
		\begin{cases}
			-\Delta v_{\delta} = \left( \alpha_{\delta} - x_1^{4} - x_1^{2}x_2^{2} - \delta x_2^{4}	\right)\mathbb{1}_{\{v_{\delta} > 0 \}} & \text{ in }\RR^{2},\\
			v_{\delta} \geq 0, \int_{\RR^{2}}v_{\delta} = 1. &
		\end{cases}
	\end{equation}
	
	Suppose $ \limsup_{\delta \rightarrow 0}\alpha_{\delta} = +\infty $. Replicating the argument in the proof of Proposition \ref{prop:nondeg_q_bound}, we get 
	\begin{equation}
		\label{eq:inf_noncoercive}
		v_{\delta} \geq C\alpha_{\delta} ~\text{ in }B_r,
	\end{equation}
for some $ r > 0 $ independent of $ \delta $. This implies $ \|v_{\delta} \|_{L^{1}(\RR^{2})} > 1$ for some $ \delta $. This is a contradiction, therefore there is some $ C_0 $ such that $ \alpha_{\delta} \leq C_0 $ for all $ \delta \in (0,1) $. 

	Consider the function
	\[\varphi(x) = \begin{cases}
		\alpha_{\delta}^{3/2} \left( \dfrac{\sqrt{5}}{3} - \dfrac{1}{2} \left( \dfrac{x_1}{\alpha_{\delta}^{1/4}} \right)^{2} +\dfrac{1}{30} \left( \dfrac{x_1}{\alpha_{\delta}^{1/4}} \right)^{6}\right) & \text{ if }  \dfrac{x_1^{4}}{\alpha_{\delta}} \leq 5, \\
		0 & \text{ otherwise,}
	\end{cases}\]
which solves 
	\[
		\begin{cases}
			-\Delta \varphi = \left( \alpha_{\delta} - x_1^{4} \right) \mathbb{1}_{\{\varphi > 0\}} \faskip \text{in } \RR^{2},
			\varphi \geq 0 & \text{in }\RR^{2}.
		\end{cases}
	\]
	In particular, $ -\Delta \varphi \geq \alpha_{\delta} - x_1^{4} - x_1^{2}x_2^{2} - (\delta x_2)^{6} $ in $ \RR^{2} $. Since $ v_{\delta} \equiv 0 $ on $ \partial B_{\delta^{-1} R} $, we can use Lemma \ref{lem:mp} in the ball $ B_{\delta^{-1}R} $ to get $ v_{\delta} \leq \varphi $ in $ B_{\delta^{-1}R} $. 

	This yields that 
	\[
		v_{\delta} \leq \frac{\sqrt{5}\alpha_{\delta}^{3/2}}{3} \faskip \text{ in } \RR^{2},
	\]
	which, together with \eqref{eq:inf_noncoercive}, implies that $ \liminf_{\delta \rightarrow 0} \alpha_{\delta} > 0 $. Moreover,
	\[
		\supp (v_{\delta})\subset (-(5C_0)^{1/4}, (5C_0)^{1/4}) \times \RR \faskip \forall \delta \in (0,1).
	\]

	We can thus take a subsequence $ \delta_n \rightarrow 0 $ such that $ \alpha_{n} \rightarrow \alpha > 0$. Moreover, the right-hand side in \eqref{eq:obstacle_delta} is bounded in $ L^{\infty} $, so that, up to a further subsequence, $ v_{\delta_n} \rightharpoonup \overline{v} $ weakly in $ W^{2,p}_{\loc}(\RR^{2}) $, with $ \overline{v} \in L^{1}(\RR^{2}) $. This implies that 
	\[ 
		v(x) = \left( \frac{\overline{\alpha}}{\alpha} \right)^{2/3} \overline{v}\left( \left( \frac{\alpha}{\overline{\alpha}} \right)^{1/4} x \right)
	\]
solves \eqref{eq:obst_noncoercive}.
\end{proof}

\begin{prop}
	\label{prop:supp_noncoercive}
	There exist $ C_0, t_1 > 0 $ such that $ \supp (U_n) \subset (-t_1, t_1) \times \RR $. Moreover, there is $ C > 0 $ such that 
	\begin{equation}
		\label{eq:infty_noncoercive}
		\|U_n\|_{L^{\infty}(\RR^{2})}\leq C_0 \frac{\beta_n}{\varepsilon_n^{4}} \faskip \forall n\in \NN.
	\end{equation}
\end{prop}
\begin{proof}
	The proof of Proposition \ref{prop:nondeg_q_bound} may be repeated verbatim in this case, which implies that $ \frac{\beta_n}{\varepsilon_n^{4}} \leq C_0 $ for some $ C_0 > 0 $. Then, $ (U_n)_{n\in \NN} $ is bounded in $ W^{2,p}_{\loc} $ for all $ p \in (1,\infty) $.

	Note that, since $ \supp (U_n) \subset B_{\varepsilon_n^{-1}R} $, we have that $ U_n $ minimizes 
	\[
		v\mapsto \int_{B_{\varepsilon_n^{-1}R}} \frac{|H_n \nabla v|^{2}}{2} + \left( y_1^{4} +y_1^{2}y_2^{2} + r_n(x) \right) v \mathrm{dy} 
	\]
	among functions all nonnegative functions $ v \in H^{1}_0(B_{\varepsilon_n^{-1}R}) $. Then, for $ n $ sufficiently large, a bump function $ \varphi $ for $ B_{1/2} $ is admissible, so that
	\[\|\nabla U_n\|_{L^{2}(\RR^{2})} \leq C_1.\]
By Nash's inequality, we have $ \|U_n\|_{L^{2}(\RR^{2})} \leq C_2 $.

Take $ \overline{x} \in B_{\varepsilon_n^{-1}R}$ such that 
\[
	1 - \frac{1}{2} x_1^{4} \leq 0 \faskip \forall x \in B_2(\overline{x}).
\]
In particular, 
\[
	-\divg(H_n\nabla U_n)\leq \left( \frac{\beta_n}{\varepsilon_n^{4}} - \frac{1}{2}(y_1^{4} + y_1^2 y_2^2 + \varepsilon_n^{2}y_2^6) \right) \mathbb{1}_{\{U_n > 0\}} \leq 0 \faskip \forall x\in B_2(\overline{x})
\]
By the local maximum principle (see \cite[Theorem 8.17]{GT83}), we have
\[
	\sup_{B_1(x_0)} U_n \leq C \| U_n \|_{L^{2}(B_2(\overline{x}))} \leq C_3.
\]

Take $ t = (2C_3 \kappa_0 + C_0)^{1/4} $, where $ \kappa_0 $ is that of Lemma \ref{lem:nondeg}, and compute
\[
	-\divg(H_n \nabla U_n) \leq -C_3 \kappa_0 \mathbb{1}_{\{U_n > 0 \}} \faskip \forall x \in (- t, t) \times \RR
\]
This yields, as in the proof of Proposition \ref{prop:compact}, that $ U_n \equiv 0 $ for $ |x_1| > \frac{3}{2}t = t_1 $.

Define 
\[\psi = \frac{\beta_n}{\varepsilon_n^{4}\inf_{B_R} h_{11}} \left( r_1^{2} - y_1^{2} \right),\]
which satisfies
\begin{align*}
	-\divg(H_n \nabla\psi) &= -h_{11}^{n} \partial_{y_1y_1}\psi - \partial_{y_1}\psi(\partial_{y_1}(h_{11}^{n}) + \partial_{y_2} (h_{12}^{n})) \\
			       &= \frac{2 \beta_n}{\varepsilon_n^{4}\inf_{B_R} h_{11}}h_{11}^{n} +\varepsilon_n \left(\frac{2\beta_nx_1}{\varepsilon_n^{4}}(\partial_1h_{11}(\varepsilon_n y) + \partial_2 h_{12}(\varepsilon_n y))\right) \\
			       &\geq \frac{2 \beta_n}{\varepsilon_n^{4}}  - \frac{2\beta_n r_1 \|H\|_{C^{1}(B_R)}}{\varepsilon_n^{4}} \varepsilon_{n} \\
			       &\geq \frac{\beta_n}{\varepsilon_n^{4}}
\end{align*}
in $ (-r_1, r_1)\times \RR$ for $ n $ sufficiently large. Apply Lemma \ref{lem:mp} in $ (-r_1, r_1)\times (-\varepsilon_n^{-1}R, \varepsilon_n^{-1}R) $ to get $ \psi \geq U_n $ in this set. In particular, 
\[U_n \leq \frac{\beta_n r_1^{2}}{\varepsilon_n^{4}\inf_{B_R} h_{11}} \faskip \text{in }B_{\varepsilon_n^{-1}R},\]
which completes the proof.
\end{proof}

\begin{prop}
	There exist $ C_1,C_2, \rho_0, n_0 > 0$ such that for all $ \rho > \rho_0 $ and $ n\geq n_0 $,
	\[\supp(U_n) \cap \{|x_2| \in (\rho, 2\rho)\}\subset \left\{|x_1| \leq \frac{C_1}{\rho}\right\} \cap \{|x_2| \in (\rho, 2\rho)\}, \]	
	Moreover, in this latter set, 
	\[ U_n \leq \frac{C_2\beta_n}{\varepsilon_n^4\rho^2}. \]
\end{prop}

\begin{proof}
	Let $ \rho > 0 $. By the symmetry of the equation, it suffices to prove the result for $ x_2 > \rho $. Denote $ Q_{\rho} = (-r_1, r_1) \times \left( \frac{\rho}{2}, \rho \right) $. In this set, we have
	\begin{align*}
		\frac{\beta_n}{\varepsilon_n^{4}} - (y_1^4 + y_1^2y_2^2) + r_n(y) & \leq \frac{\beta_n}{\varepsilon_n^{4}} - \frac{1}{2} y_1^2 y_2^2 \\ 
										  &\leq C_1 - \frac{\rho^2}{8} y_1^2 ,
	\end{align*}
for some $ \frac{\beta_n}{\varepsilon_n^4} \leq C_1 $. 

Denote $ c_1 = \inf_{B_R} h_{11} $, $ c_2 = \sup_{B_R}h_{22} $ and define the following functions
\begin{align*}
	\varphi_1(x_1) & = \begin{cases}
		\frac{1}{c_1} \left( \frac{24C_1^2}{\rho^2} - C_1 x_1^2 + \frac{\rho^2}{96} x_1^4 \right)  & \text{ if } x_1^2 \leq \frac{48C_1}{\rho^2}, \\
		0 & \text{ otherwise},
	\end{cases} \\ 
		\varphi_2(x_2) & = \begin{cases}
			\dfrac{C_1}{8r_2} \left( x_2 - \left( \frac{\rho}{2} +\sqrt{8C_0r_2} \right)  \right)^2 & \text{ if } x_2 - \frac{\rho}{2} \in  \left( 0, \sqrt{8C_0r_2} \right), \\ 
			0  & \text{ otherwise },
		\end{cases} \\ 
		\varphi_3(x_2) & = \begin{cases}
			\dfrac{C_1}{8r_2} \left( x_2 - \left( 3\rho -  \sqrt{8C_0r_2} \right)  \right)^2 & \text{ if } 3\rho - x_2 \in  \left( 0,\sqrt{8C_0r_2} \right), \\ 
			0  & \text{ otherwise },
		\end{cases}
\end{align*}
and $ w (y_1, y_2) = \varphi_1(y_1) + \varphi_2(y_2) + \varphi_3(y_2) $. We compute
\begin{align*}
	-\divg(H_n \nabla \varphi_1) &= \frac{1}{\inf_{B_R}h_{11}} \left( h_{11}^n \left( 2C_1 - \frac{\rho^2}{8} y_1^2 \right) +\varepsilon_n \left( 2C_1 y_1 - \frac{\rho^2}{24} y_1^3 \right)(\partial_1 h^n_{11}(\varepsilon_n y) + \partial_2 h^n_{12}(\varepsilon_n y)) \right)\mathbb{1}_{\{ \varphi_1 > 0\}} \\
				     &\geq 2C_1 - \frac{\rho^2}{8}y_1^2 - \frac{\varepsilon_n}{\rho}\left(2C_1 + \frac{1}{24}\right)\|H\|_{C^1(B_R)}\\
				     &\geq 2C_1 - \frac{\rho^2}{8}y_1^2 - \varepsilon_n\left(2C_1 + \frac{1}{24}\right)\|H\|_{C^1(B_R)}\\
				     &\geq \frac{7C_1}{4} - \frac{\rho^2}{8}y_1^2
\end{align*}
for $ n $ large enough but independent of $ \rho $. Moreover, for $ x_2 \in \left( \frac{\rho}{2}, \frac{\rho}{2} + \sqrt{8C_0r2} \right)  $ 
\begin{align*}
	-\divg(H_n \nabla \varphi_2) & = -h_{22}^n \frac{C_1}{4r_2} - \varepsilon_n	\frac{C_1}{4r_2} \left( x_2 -\left( \frac{\rho}{2} + \sqrt{8C_0r_2} \right) \right) \left(\partial_1 h_12^n(\varepsilon_n y) + \partial_2 h_22^n(\varepsilon_n y) \right) \\
				     & \geq -\frac{C_1}{4} - \varepsilon_n\sqrt{8C_0r_2} \|H\|_{C^1(B_R)}\\
				     & \geq -\frac{3C_1}{8},
\end{align*}
for $ n $ large independently of $ \rho $. The computation for $ \varphi_3 $ in the range $ x_2 \in (3\rho - \sqrt{8C_0r_2}, 3\rho) $ is analogous. We conclude
\[
	-\divg(H_n \nabla w) \geq C_1 - \frac{\rho^2}{8}y_1^2 \faskip \text{ in }Q_{\rho}.
\]

Since, by construction, $ w\geq 0  $ in $ Q_t $ and $ w \geq U_n $ on $ \partial Q_t $, we apply Lemma \ref{lem:mp} to get $ w \geq U_n $. Then, $ U_n \equiv 0 $ if $ x_1 > \frac{\sqrt{48C_1}}{\rho} $, and
\[ U_n \leq \frac{24 C_1^2}{\rho^2 \inf_{B_R}h_{11}} \faskip \text{ for } x_2 \in \left( \frac{\rho}{2} + \sqrt{8C_0r_2}, 2\rho - \sqrt{8C_0r_2}  \right). \]
The result follows by taking $ \rho_0 \geq 2 \sqrt{8C_0r_2} $. 
\end{proof}

\begin{proof}[Proof of Theorem \ref{mainthm:noncoercive}]
	We know that $ \limsup_{n\rightarrow \infty} \frac{\beta_n}{\varepsilon_n^{4}} < \infty $. Then, the right-hand side of the equation for $ U_n $ is bounded in compact sets. Replicating the argument of the proof of Theorem \ref{mainthm:isolated_hom}, there exists $ \overline{U} \in W^{2,p}_{\loc}(\RR^{2}) $ such that $ U_n \rightharpoonup \overline{U} $ weakly in $ W^{2,p}_\loc(\RR^{2}) $ and a.e. in $ \RR^{2} $.
	
	Take $ \rho_k = 2^k \rho_0 $ for $ k\in \NN $ and define 
	\begin{align*}
		Q_0 &= (-r_1, r_1) \times (-\rho_0, \rho_0), \\ 
		Q_{k+1} &= \left( -\frac{C}{\rho_k}, \frac{C}{\rho_k} \right) \times ((-\rho_{k+1}, -\rho_k) \cup (\rho_k, \rho_{k+1})), \\
		\Psi(y) &= T \sum_{k=0}^{\infty}\frac{1}{\rho_k^2} \mathbb{1}_{Q_k} 
	\end{align*}
	where $ T $ is chosen so that $ C_0 \frac{\beta_n}{\varepsilon_n^{4}} \leq T$ in \eqref{eq:infty_noncoercive}. Then, $ U_n \leq \Psi $ for all $ n $ sufficiently large. Moreover, 
	\begin{align*}
	\int_{\RR^{2}} \Psi = T \left( 4r_1\rho_0 + \sum_{k=0}^{\infty}\frac{4C(\rho_{k+1} - \rho_k)}{\rho_k^{3}} \right) \leq C \sum_{k=0}^{\infty} \frac{1}{2^{2k}} < \infty.
	\end{align*}

	By the dominated convergence theorem, we have that 
	\[ 
		1 = \int_{\RR^{2}} U_n \mu_n(y) \mathrm{d}y \rightarrow \int_{\RR^{2}} \overline{U} \mathrm{d}y,
	\]
	which implies $\| \overline{U} \|_{L^{1}(\RR^{2})} = 1$. 

		Hence, by \eqref{eq:infty_noncoercive}, we must have $ \liminf_{n\rightarrow \infty} \frac{\beta_n}{\varepsilon_n^{4}} > 0 $. Therefore, we may choose a further subsequence such that $ \frac{\beta_n}{\varepsilon_n^{4}} \rightarrow \overline{\alpha} > 0 $ and $ \overline{U} $ solves \eqref{eq:obst_noncoercive}.  
	\end{proof}

	\section{Miscellanea}
	\label{sect:misc}

	The approach used in the present paper may be used to treat a variety of other degenerate cases. In this section, we comment on other cases and point out some limitations of this approach.

	\subsection{Maximum attained on a curve}

	The analysis performed in Section \ref{sect:inhom} may be translated to a case where $ g $ attains its maximum on a closed curve. We describe how to do this in the sequel. Note that while we assume certain nondegeneracy of $ g $ along the whole curve, one could also study degenerate cases, where some concentration phenomena (akin to the one described after Theorem \ref{mainthm:isolated_hom}) is expected.

	Let $ \Gamma = \mathbb{S}^{2}$, and let $ \gamma \subset \Gamma$ be a circle in $ \Gamma $. Parametrize $ \gamma $ by $ s:[0,2\pi] \rightarrow \Gamma $ and write, in coordinates,
	\[s(t) = \left(\sin(\phi_0)\cos \left( t \right) , \sin(\phi_0)\sin\left( t \right), \cos(\phi_0)\right), \]
	for $ \phi_0 \in \left(0, \frac{\pi}{2}\right) $. Assume that $ g $ attains a maximum 
	on $ \gamma $, and that 
	\[g(\phi, t) = \gmax - a(t) (\phi - \phi_0)^{2} + O(|\phi - \phi_0|^{3}) \leq \gmax - \frac{a(t)}{2}(\phi - \phi_0)^{2}\]
	if $ |\phi - \phi_0| \leq \frac{1}{4} $with $ \inf a(t) = a_0 > 0 $. Suppose, without loss of generality, that there are no other maxima for $ \phi \in \left(\frac{3}{4}\phi_0, \frac{5}{4}\phi_0\right) $

	Let $ M_n \rightarrow 0 $, and take $ n $ sufficiently large in order to have $ u_n(\phi_0 \pm \frac{\phi_0}{4}, t) = 0 $ for all $ t \in [0,L]$, denote $ \Theta \subset \Gamma$ as the set delimited by the curves $ \{\phi = \left( 1 \pm \frac{1}{4} \right) \phi_0 \} $, and $ m_n $ the mass of $ u_n $ over $ \Theta $.

	Define the rescaling
	\[u_m(\phi, t) = m_n^{4/5}U_n(\phi_0 + m_n^{-1/5} \eta, t),\]
	which satisfies
	\begin{align*}
		\int_{\Theta} u_m \mathrm{d}x &= \int_{0}^{2\pi} \int_{3\phi_0/4}^{5\phi_0/4} u_m(\phi, t) \sin(\phi) \mathrm{d}\phi \mathrm{d}t	\\
					      &= \int_{0}^{2\pi} \int_{-m_n^{-1/5}\phi_0/4}^{m_n^{-1/5}\phi_0/4} m_n U_n(t, \eta) \sin \left( \phi_0 + m_n^{1/5} \eta \right) \mathrm{d}\eta \mathrm{d}t,
	\end{align*}
	and hence
	\[
		 \int_{0}^{2\pi} \int_{-m_n^{-1/5}\phi_0/4}^{m_n^{-1/5}\phi_0/4} U_n(t, \eta) \sin \left( \phi_0 + m_n^{1/5} \eta \right) \mathrm{d}\eta \mathrm{d}t = 1.
	\]
	In what follows, $ \varepsilon_n := m_n^{1/5} $.

	Moreover, denoting
	\[
		L_n V =  \frac{\partial^{2} V}{\partial \eta^{2}} + \varepsilon_n \text{cotan}(\phi_0 + \varepsilon_n\eta) \frac{\partial V}{\partial \eta} + \frac{\varepsilon_n^2}{\sin^2(\phi_0 + \varepsilon_n \eta)} \frac{\partial^{2} V}{\partial s^2},
	\]
	we may write \eqref{eq:obstacle_gamma} for $ U_n $ as
	\[
		-L_n	U_n = \left( \frac{\beta_n}{\varepsilon_n^2} - a(t) \eta^2 +r_n(\eta)\right) \mathbb{1}_{\{U_n > 0 \}} \faskip \text{ in } \left(-\frac{1}{4\varepsilon_n}, \frac{1}{4\varepsilon_n}\right) \times 2\pi \mathbb{S}^1
	\]

	In the formal limit, we have
	\[
		\begin{cases}
			-\partial^{2}_{\eta \eta} \overline{U} = \left( \overline{\alpha} - a(t) \eta^2  \right) \mathbb{1}_{\{\overline{U} > 0\}} \faskip \text{ in }\RR \times 2\pi \mathbb{S}^{1},\\ 
			\overline{U} \geq 0,~~ \int_{0}^{2\pi} \int_{\RR} \overline{U}(\eta, t) \mathrm{d}\eta \mathrm{d}t = \frac{1}{\sin(\phi_0)}.
		\end{cases}
	\]
	By the argument in Section \ref{ssect:limit_deg}, we readily see that the solution $ (\overline{U}, \overline{\alpha}) $ to this problem is unique and has the form
	\[
		\overline{U}(\eta, t) = \begin{cases}
			\frac{3 \overline{\alpha}^2}{2a(t)} -\frac{\overline{\alpha}}{2}\eta^2	+\frac{a(t)}{12} \eta^4 & \text{ if } \frac{a(t) \eta^2}{3} \leq \overline{\alpha},\\
			0 & \text{ otherwise,}
	\end{cases}
	\]
	with $ \overline{\alpha} $ chosen in order to satisfy the mass restriction, namely
	\[\overline{\alpha} = \left( \frac{23 \sqrt{3} \sin(\phi_0)}{20} \int_{0}^{2\pi} a^{-3/2}(t)\mathrm{d}t \right)^{-2/5}. \]
	By the assumption $ a \geq a_0 > 0 $, $ \overline{U} $ is well defined.

	Note that, in this case, it is possible to compare $ U_n $ with the subsolution
	\[\psi(\eta, t) = \frac{\beta_n}{4\varepsilon_n^2} (r_0^{2} - \eta^2)\]
	for $ r_0 > 0 $ depending only on $ a $ and a lower bound for $ \frac{\beta_n}{\varepsilon_n^{2}} $. As in Proposition \ref{prop:deg_bound_q}, this implies $ \limsup_n \frac{\beta_n}{\varepsilon_n^2} < \infty$.

	The rest of the discusion in Section \ref{sect:inhom} may be carried out verbatim to get
	\[U_n \rightarrow \overline{U} \faskip \text{ in } L^{\infty}(\RR \times 2\pi\mathbb{S}^1)\]
	and 
	\[\partial_{\eta}U_n \rightharpoonup \partial_{\eta}\overline{U} \faskip \text{ weakly in } L^{2}(\RR \times 2\pi\mathbb{S}^1).\]

\subsection{A heuristic argument for other polynomials}

The analysis performed in Section \ref{sect:noncoercive} may be useful in other cases where the level sets of the homogeneous part of $ g $ lead to $ L^1 $ estimates. However, the method breaks down if the homogeneous part $ \gmax - g $ has a zero of sufficiently high degree in a given direction. 

Consider for example
\[g(x) = \gmax - x_1^8 - x_1^6 x_2^2 + r(x).\]
This implies, as in Section \ref{sect:noncoercive}, that, if $ x_2 > \rho $, then 
\[
	\begin{array}{cc}	
		U_m = 0, \faskip \forall |x_1|>\frac{C}{\rho^{1/3}}, \text{ and } & U_m \leq \frac{C}{\rho^{2/3}}, \faskip \forall |x_1| \leq \frac{C}{\rho^{1/3}},
	\end{array}
\]
which implies that, with $ Q_k $ defined analogously as in the proof of Theorem \ref{mainthm:noncoercive}, we define
\[\Psi = T \sum_{k=0}^{\infty} \frac{1}{\rho_k^{2/3}} \mathbb{1}_{Q_k}\]
which implies 
\[
	\int_{\RR^{2}} \Psi \geq C \sum_{k}^{\infty} \frac{\rho_{k+1} - \rho_k}{\rho_k} = C \sum_{k}^{\infty} 1 = \infty,
\]
so that the dominated convergence theorem cannot be applied.

In a general setting, there is a simple heuristic giving way to an implicit description of the limit profile. Suppose that a higher order Taylor expansion yields
\[g(x) = \gmax - x_1^8 - x_1^6 x_2^2 - x^{10} + r(x).\]
Note that, in this case, the homogeneous polynomial of least degree has a nontrivial zero set (i.e. $ \{x_1 = 0\} $), but the higher order term forces the maximum at $ x=0 $ to be strictly positive.

Define the coordinates
\[
	\begin{array}{cc}
		x_1 = \ell y_1, & x_2 = x_{2,0} + \ell y_2,		
	\end{array}
\]
where $ x_{2,0}, \ell > 0 $ may depend on $ \beta_{M} $. Moreover, we will require that $ \ell \ll x_{2, 0} $. 

Using the same notation as in the previous sections, we compute
\[
	(1 + \alpha_M)g - 1 = \beta_M - x_{2,0}^{10} -\ell^{6} x_{0,2}^{2} y_1^6 \left( \left(\frac{\ell}{x_{0,2}^2} y_1 \right)^2+ \left( 1 + \frac{\ell}{x_{0,2}} y_2 \right)^2 \right) + o(1).
\]
Choose $ x_{0,2} = t \beta_{M}^{1/10} $ with $ t \in (0,1) $ fixed and $ \ell = \beta_M^{2/15} $. Hence
\[
	(1 + \alpha_M)g - 1 = \beta_M \left( 1 - t^{10} - t^{2}y_1^6  + o(1) \right) 
\]
this suggests the rescaling
\[
	u_M(x) = \beta_M^{19/15} U_M(\beta_M^{-2/15} x_1, \beta_M^{-2/15}(x_2 - t\beta_M^{1/10})),
\]
which satisfies
\[
	-L_M U_M = \left( 1 - t^{10} - t^2y_1^6 +o(1) \right)\mathbb{1}_{\{U_M > 0\}},
\]
with $ -L_M\varphi \rightarrow -\Delta \varphi $ for any $ \varphi \in C^{2}_c(\RR^2) $. 

Assuming that $ U_M $ is bounded in $ L^{\infty}(\RR^{2}) $, using nondegeneracy and regularity estimates we get that the support of $ U_M $ is bounded in the $ y_1 $ direction and that it converges, up to a subsequence and weakly in $ W^{2,p}_{\loc}(\RR^{2}) $ to a bounded solution of
\[-\Delta \overline{U} = \left( 1- t^{10} - t^2 y_1^6 \right) \mathbb{1}_{\{\overline{U} > 0\}}.\]
This equation may be solved explicitly. If $|t| > 1$, then $\overline{U} = 0$. If $|t| < 1$, by a computation similar to that of Section \ref{sect:inhom} get
\[
\overline{U}(y) =
\begin{cases}
\dfrac{(1-t^{10})^{4/3}}{t^{2/3}} \left( \dfrac{3 \sqrt[3]{7}}{8}- \dfrac12 \left( \dfrac{t^{1/3}y_1}{(1-t^{10})^{1/6}}\right)^2 + \dfrac{1}{56}\left( \dfrac{t^{1/3}y_1}{(1-t^{10})^{1/6}}\right)^8\right) & \text{ if }	\left( \dfrac{t^{1/3}y_1}{(1-t^{10})^{1/6}}\right)^6 \leq 7 \\
0  &\text{ otherwise}.
\end{cases}
\]

Reverting back to the $x$ variable yields several consequences. First, if $t$ is of order $\beta_M^{1/10}$, then
\[ u_M(0,t) \sim \frac{(1-t^{10})^{4/3}}{t^{2/3}},\]
for $M$ sufficiently small. If $x_2 > t \beta_M^{1/10}$, then
\[u_M \equiv 0 \faskip \forall |x_1| \geq \frac{C}{\beta_M^{2/15}t^{1/3}}, \]
which shows the asymptotics of the width of the support in the $x_1$ direction.

\textbf{Acknowledgments} The authors gratefully acknowledge the financial support of the Deutsche Forschungsgemeinschaft (DFG, German Research Foundation) through the collaborative research centre "Analysis of criticality: from complex phenomena to models and estimates" (CRC 1720, Project-ID 539309657) and the Bonn International Graduate School of Mathematics at the Hausdorff Center for Mathematics (EXC 2047/2 Project-ID 390685813). The authors would like to thank Marvin Weidner for valuable comments on a preliminary version of the paper.

\printbibliography

\end{document}